\documentclass[11pt,reqno]{amsproc}

\usepackage[margin=1in]{geometry}

\usepackage{comment}

\usepackage{esint,stackrel}
\usepackage{enumitem}
\usepackage{bbm}

\usepackage{graphicx}
\usepackage{amsmath,amsthm,amsfonts,amssymb,amscd}

\usepackage[colorlinks=true, pdfstartview=FitV, linkcolor=blue,citecolor=blue, urlcolor=blue]{hyperref}
\usepackage[usenames,dvipsnames]{color}
\usepackage{times}

\title[On Moffatt's magnetic relaxation equations]{On Moffatt's magnetic relaxation equations}

\author[R.~Beekie]{Rajendra Beekie}
\address{Courant Institute of Mathematical Sciences, New York University, New York, NY 10012}
\email{beekie@cims.nyu.edu}

\author[S.~Friedlander]{Susan~Friedlander}
\address{Department of Mathematics, University of Southern California, Los Angeles, CA 90089}
\email{susanfri@usc.edu}

\author[V.~Vicol]{Vlad Vicol}
\address{Courant Institute of Mathematical Sciences, New York University, New York, NY 10012}
\email{vicol@cims.nyu.edu}

\def\p{\partial}
\def\eps{\varepsilon}
\def\les{\lesssim}
\renewcommand*{\div}{\ensuremath{\mathrm{div\,}}}
 
\newcommand{\brak}[1]{\bigl[ #1 \bigr]} 
\newcommand{\norm}[1]{\left \| #1 \right\|} 
 
\newcommand{\abs}[1]{\left|#1\right|}

\newcommand{\RR}{\mathbb R}
\newcommand{\Proj}{\mathbb P}

\newcommand{\TT}{\mathbb T}
\newcommand{\OO}{\mathcal O}

\renewcommand*{\bar}{\overline}

\newtheorem{theorem}{Theorem}[section]
\newtheorem{lemma}[theorem]{Lemma}
\newtheorem{prop}[theorem]{Proposition}

\theoremstyle{definition}

\newtheorem{remark}[theorem]{Remark}

\numberwithin{equation}{section}

\newcommand{\po}{\mathbb{P}_{0}}
\newcommand{\pn}{\mathbb{P}_{\! \perp}}

\date{\today}

\begin{document}
\maketitle

\begin{abstract}
We investigate the stability properties for a family of equations introduced by Moffatt to model magnetic relaxation. These models preserve the topology of magnetic streamlines, contain a cubic nonlinearity, and yet have a favorable $L^2$ energy structure. We consider the local and global in time well-posedness of these models and establish a difference between the behavior as $t\to \infty$ with respect to weak and strong norms.
\end{abstract}

\setcounter{tocdepth}{1} 
\tableofcontents 

\vspace{-0.33in}

\section{Introduction}
In the 1960s  V.I.~Arnold developed a new set of geometric ideas concerning the incompressible Euler equations governing the flow of an ideal fluid. In the following decades the subject of  topological  hydrodynamics  flourished. Following Arnold's seminal work~\cite{Arnold66} there was an enormous body of literature on the subject. We refer to only a few of the many important papers
including Ebin and Marsden~\cite{EbinMarsden70}, Holm, Marsden, Ratiu, and Weinstein~\cite{HolmEtAl85}, and Arnold and Khesin~\cite{ArnoldKhesin98}. This geometric perspective  views  the incompressible Euler equations as the geodesic equations of a right-invariant metric on the  infinite-dimensional group of volume preserving diffeomorphisms. Of particular importance are the fixed points of the underlying dynamical system, namely {\em steady fluid flows}, and their topological richness~\cite{EncisoPeraltaSalas12,Gavrilov19,ConstantinLaVicol19}. Moreover, as with any dynamical system, of fundamental importance is the question of {\em accessibility} of these equilibria. In this paper, we discuss a mechanism of reaching these equilibria not through the Euler vortex dynamics itself, but via a topology preserving diffusion process, called   {\em magnetic relaxation}.

The magnetic relaxation equation (MRE) considered here was introduced by Moffatt~\cite{Moffatt85,Moffatt21}  to describe a topology-preserving dissipative equation, whose solutions are conjectured to converge  in the infinite time limit  towards ideal Euler/magnetostatic equilibria (see also Brenier~\cite{Brenier14}); we recall the motivation in Section~\ref{sec:motivation} below. We are interested in understanding the long time behavior for 
\begin{subequations}
\label{eq:MRE}
\begin{align}
\p_t B + u \cdot \nabla B &= B \cdot \nabla u 
\label{eq:B:evo} \\
(-\Delta)^\gamma u &=B \cdot \nabla B + \nabla p 
\label{eq:u:def} \\
\div u = \div B &=0
\label{eq:div:free}
\end{align}
\end{subequations}
where the unknowns are the incompressible velocity vector field $u$, the magnetic vector field $B$, and the fluid pressure $p$. We consider the problem  posed on $\TT^d = [- \pi ,  \pi ]^d$ with $d\in \{2,3\}$, and $u$ is taken to have zero mean on $\TT^d$. The parameter $\gamma \geq 0$ is a  regularization parameter of the constitutive law $B \mapsto u$: the case $\gamma = 0$ corresponds to a Darcy-type regularization (as was done in~\cite{Moffatt85,Brenier14,Moffatt21}), the case $\gamma = 1$ corresponds to a Stokes-type regularization, while the general case $\gamma>0$ may be alternatively used in numerical simulations to smoothen the velocity gradients. This constitutive law may be written as 
\begin{align}
u  = (-\Delta)^{-\gamma} \Proj ( B \cdot \nabla B) = (-\Delta)^{-\gamma} \Proj \div ( B \otimes B)
\label{eq:B:to:u}
\end{align}
where $\Proj$ is the Leray projector (onto divergence-free vector fields).
We emphasize that the topology of the vector field $B$ is preserved under the vector transport equation \eqref{eq:B:evo} irrespective of the regularization parameter $\gamma$ in the constitutive law~\eqref{eq:u:def}. We note that if $\div B_0 = 0$, then the vector transport equation \eqref{eq:B:evo} preserves the incompressibility of $B$ at all later times.

From a mathematical perspective, the analysis of the MRE system~\eqref{eq:MRE} is unusually challenging. Not only is it an active vector equation, versus the more familiar active scalar equations in fluid dynamics~\cite{ConstantinMajdaTabak94,CastroCordobaGancedoOrive09}, but the nonlinearity is cubic in $B$. Some of the interesting special features of MRE are discussed in the article of Brenier~\cite{Brenier14}. Brenier presents a concept of dissipative weak solutions for MRE when the regularization parameter $\gamma$ is set to zero. It is shown in two space dimensions that the initial value problem admits such global dissipative weak solutions, and that they are unique whenever they are smooth. However, not even the local existence of strong solutions to~\eqref{eq:MRE} is known.

Besides local well-posedness, in this paper we examine the long-time behavior of the magnetic relaxation equations \eqref{eq:MRE}, and show that although the velocity field $u(\cdot,t)$ converges to $0$ as $t\to \infty$ (for a sufficiently large regularization parameter $\gamma$), there are many open questions regarding the sense in which the magnetic field $B(\cdot,t)$ itself converges as $t\to \infty$ ({\em weak} vs {\em strong} convergence; see Remark~\ref{rem:relaxation?}). In two dimensions, we give a specific example of asymptotic stability to a simple two dimensional steady state. In contrast, for a specific class of two-and-a-half-dimensional solutions we illustrate instability for the MRE system \eqref{eq:MRE}, by showing that the  the magnetic current $\nabla \times B$ grows unboundedly as $t\to \infty$. Our results are presented in Section~\ref{sec:results} below.

\subsection{Motivation behind magnetic relaxation}
\label{sec:motivation}
There are certain well known analogies between Euler equilibria and equilibria of incompressible magnetohydrodynamics (MHD). Recall that the ideal incompressible MHD equations are 
\begin{subequations}
\label{eq:MHD}
\begin{align}
\p_t B + u\cdot \nabla B &= B \cdot \nabla u \label{eq:MHD:a} \\
\p_t  u + u \cdot \nabla u  + \nabla p&= B \cdot \nabla B \label{eq:MHD:b}  \\
\div u = \div B &= 0 \,.
\end{align} 
\end{subequations}
 The equilibrium equation for magnetostatics is obtained by setting $B = \bar B(x)$ and $u = 0$ in \eqref{eq:MHD} to give
\begin{align}
\nabla \bar p  = \bar B  \cdot \nabla \bar B \,, \qquad \div \bar B  = 0 \,. 
\label{eq:steady:MHD}
\end{align}
In comparison, the equilibrium equation for incompressible Euler steady states is obtained by setting $u = \bar u(x)$ and $B =0$ in \eqref{eq:MHD} to give
\begin{align}
\bar u \cdot \nabla \bar u + \nabla \bar p =0\,, \qquad \div \bar u = 0 \,. 
\label{eq:steady:Euler}
\end{align}
Clearly, any vector field $\bar B$ that satisfies \eqref{eq:steady:MHD} is also an equilibrium solving \eqref{eq:steady:Euler} upon changing the sign of the pressure, and vice-versa. However, this analogy between magnetic and fluid steady states \eqref{eq:steady:MHD}--\eqref{eq:steady:Euler} does not extend to the evolution of perturbations about these steady states, as governed by the ideal MHD system on the one hand, respectively the pure Euler dynamics on the other hand. For example, stability issues for Euler steady flows are not the same as the stability for magnetostatic equilibria~\cite{HolmEtAl85,Moffatt86,SadunVishik93,FriedlanderVishik95}.

In \cite{Arnold74}, Arnold suggested a process which demonstrates the existence of an Euler equilibrium that has the same topological structure as an arbitrary divergence free magnetic field. The idea is to use the evolution dynamics of the magnetic field to reach an Euler/magnetic equilibria which preserves Kelvin circulation. This concept was developed by Moffatt~\cite{Moffatt85} (see also the excellent recent overview~\cite{Moffatt21}).  The {\em magnetic relaxation} procedure envisioned by Moffatt preserves the streamline topology of an initial divergence free three-dimensional vector $B_0(x)$,  but abandons the constraint that $B(x,t)$ should remain smooth as $t\to \infty$. In this model, the magnetic field evolves under the {\em frozen field} equation \eqref{eq:MHD:a} via a vector field $u(x,t)$ which is related to $B(x,t)$ by a suitable constitutive law, which has two properties: that $u(x,t)$ formally decays to $0$ as time goes to infinity, and the vector fields $u$ and $j \times B$ are parallel with non-negative proportionality factor (here $j = \nabla \times B$ is the current field). Moffatt introduced the concept of {\em topological accessibility} which is weaker than {\em topological equivalence}\footnote{Here, we say that $B_1$ and $B_0$ are topologically equivalent, if $B_1(X(\alpha)) = \nabla_\alpha X(\alpha) B_0(\alpha)$ for a volume preserving diffeomorphism $\alpha \mapsto X(\alpha)$. In contrast, to say that $B_1$ is topologically accessible from $B_0$ means that  (see e.g.~in~\cite[Section 8.2.1]{Moffatt21}) $B_1 = \lim_{t\to \infty} B(\cdot,t)$, where $B$ is a solution of \eqref{eq:MHD:a} with initial datum $B_0$ and some solenoidal vector field $u$, under the additional property that $\int_0^\infty \left| \int_{\TT^d} B \cdot (B\cdot \nabla u) dx \right| dt < \infty$ .} because it allows for the appearance of discontinuities in the magnetic field (current sheets) as $t\to \infty$. As an example of a constitutive law relating $u$ to $B$, Moffatt~\cite{Moffatt85,Moffatt21}, also Brenier~\cite{Brenier14}, suggested 
\begin{align*}
u = B\cdot\nabla B + \nabla p \,,
\end{align*}
which may be used in conjunction with \eqref{eq:MHD:a} to show that the magnetic energy satisfies 
\begin{align*}
\frac 12 \frac{d}{dt} \norm{B}_{L^2}^2 = - \norm{u}_{L^2}^2\,. 
\end{align*}
Hence the energy of $B$ is strictly monotonically decreasing, until $u\equiv 0$. Note that $\|B(\cdot,t)\|_{L^2}$ is bounded from below uniformly in time, solely in terms of the initial magnetic helicity~\cite{Arnold74}  (see Remark~\ref{rem:helicity}).

\subsection{Main Results}
\label{sec:results}
The main results in this paper are as follows.

In Section~\ref{sec:local:Sobolev} we prove local existence for solutions of the MRE system~\eqref{eq:MRE} in Sobolev spaces $H^s$. This result holds for any $\gamma\geq 0$ and dimension $d \geq 2$, for Sobolev exponents $s >  d/2 + 1$. Theorem~\ref{thm:local} follows from the dissipative nature of \eqref{eq:MRE}, exhibited in its $L^2$ energy estimate, by using two commutator estimates at the level of $H^s$.

In Section~\ref{sec:global:Sobolev} we prove global existence in $H^s$ when the regularization parameter satisfies $\gamma > d/2 + 1$. For such $\gamma$, Theorem~\ref{thm:global:regularized} shows that the magnetic relaxation question is well-posed, because we can speak of a global in time solution. We recall that the natural values of $\gamma$ coming from physical arguments are $\gamma = 0$ (corresponding to a Darcy-type approximation) or $\gamma = 1$ (corresponding to a Stokes-type approximation); unconditional global existence in this range of $\gamma$ remains open.

In Section~\ref{sec:t:to:infinity} we investigate the possible behavior of the solutions in Section~\ref{sec:global:Sobolev}  as $t\to \infty$.
We prove in Theorem~\ref{thm:velocity:relaxation} that the velocity field $u (\cdot, t)$ converges to $0$, strongly, asymptotically as time diverges. The specific form of this relaxation is given by~\eqref{eq:velocity:relaxation}. We note, however, we do not obtain a rate for the convergence.  Also, it remains open to prove that the vector $B ( \cdot, t )$ itself converges to a steady Euler (magnetic) weak solution.

In Section~\ref{sec:2D:stability} we consider the MRE system for $d=2$ and $\gamma = 0$. We study the
asymptotic stability of a special magnetostatic state, $\bar B = e_1$, under Sobolev smooth perturbations. The evolution equation for the perturbations is given by \eqref{eq:magnetic:perturbation2}, which is an active vector equation with a cubic nonlinearity.  Equation~\eqref{eq:magnetic:perturbation2} has some similarities with the equation for the perturbation  of a linearly stratified density in the two dimensional
incompressible porous media equation (IPM); the former being, however, an active scalar equation with a quadratic nonlinearity.
In the context of IPM, Elgindi~\cite{Elgindi17} studied the asymptotic stability of the same special steady state and proved that solutions must converge
(i.e.~relax) as $t\to \infty$ to a stationary solution of the IPM equation; see also the work~\cite{CastroCordobaLear19} in the case of a bounded domain. In Theorem~\ref{theorem:main} we employ some of these ideas to prove asymptotic stability (relaxation) of MRE in this special two dimensional setting.

In Section~\ref{sec:3D:instability} we turn to the three dimensional MRE system. We observe that there is an interesting class of exact solutions to  \eqref{eq:MRE} when $\gamma = 0$, which has analogies to the well know exact solutions of the three dimensional Euler equation, which are in fact two-and-a-half dimensional, cf.~Yudovich~\cite{Yudovich74} or~DiPerna and Majda~\cite{DiPernaMajda87a}. 
In the case of the Euler equation the construction of the exact solution is based on a non-constant coefficient transport equation, which produces a two-and-a-half  dimensional flow whose vorticity grows unboundedly in time (linearly in time for shear flows~\cite{Yudovich74,Yudovich00}, or exponentially in time for cellular flows~\cite{ElgindiMasmoudi20}). In contrast, for the MRE system the construction of the exact solution is based on a non-constant coefficient heat-type equation, which has a {\em rank $1$ diffusion matrix}. By choosing the spatial dependence of the initial data appropriately, in Theorem~\ref{thm:growth} we construct an example of a magnetic field $B(\cdot,t)$ which converges (relaxes) in $L^2$ as $t\to \infty$ to a steady solution $\bar B$, but this limiting solution is not smooth and exhibits magnetic current sheets; as such,  the current $j(\cdot,t) = \nabla \times B(\cdot,t)$ grows as $t^{\frac{1}{4}}$ in $L^2$.  Additionally, in Theorem~\ref{thm:exponential:growth} we show that in the presence of hyperbolic dynamics, for instance along the separatrix of a cellular flow, the current $j(\cdot,t)$ may even grow exponentially in time, for all time, which is a strong type of instability. 

Clearly relaxation of the MRE system~\eqref{eq:MRE} is a very subtle matter.
We further illustrate this in~Section~\ref{sec:open:problems}, where we discuss a number of open problems.
 
\subsection*{Acknowledgements} 
R.B. was supported by the NSF Graduate Fellowship Grant~1839302. 
S.F. was in part supported by NSF grant DMS~1613135.
S.F. thanks IAS for its hospitality when she was a Member in 2020--21.
V.V. was in part supported by the NSF grant CAREER DMS~1911413.
V.V. thanks B.~Texier and S.~Shkoller for stimulating discussions.

\section{Local existence in Sobolev spaces for all $\gamma \geq 0$}
\label{sec:local:Sobolev}

The dissipative nature of \eqref{eq:MRE}, already alluded to in the introduction, is seen by inspecting the magnetic energy estimate
\begin{align}
 \frac 12 \frac{d}{dt} \norm{B}_{L^2}^2 
 &= \int_{\TT^d} B \cdot \left(  B\cdot \nabla  u \right)  = -  \int_{\TT^d} u \cdot \left( B \cdot \nabla  B\right)\notag\\
 &= -  \int_{\TT^d} u \cdot \left( (-\Delta)^\gamma u - \nabla p \right)    = - \norm{u}_{\dot{H}^\gamma}^2
 \, .
 \label{eq:energy}
\end{align}
Integrating in time, we deduce that 
\begin{align}
\sup_{s\in[0,t]} \norm{B(\cdot,s)}_{L^2}^2 + 2 \int_0^t \norm{u(\cdot,s)}_{\dot{H}^\gamma}^2 ds \leq \norm{B_0}_{L^2}^2 
\label{eq:energy:a}
\end{align}
for all $t>0$ such that the solution is sufficiently smooth on $[0,t]$ to justify the integration by parts manipulations in \eqref{eq:energy}.

\begin{remark}[\bf A global lower bound for the magnetic energy]
\label{rem:helicity}
We note that no matter the level of regularization in the constitutive law $B \mapsto u$ in \eqref{eq:u:def}, the magnetic helicity 
\begin{align*}
{\mathcal H}(t) = \int_{\TT^d} A(x,t) \cdot B(x,t) dx \,,
\end{align*}
is still a constant function of time,\footnote{Note in contrast that the cross-helicity $\int_{\TT^d} u\cdot B dx$ is expected to vanish as $t\to \infty$ since $B(\cdot,t)$ remains uniformly bounded in $L^2$, while $u(\cdot,t)\to 0$ in $L^2$.}  as long as the solutions remain sufficiently smooth. Here we have denoted by $A$ the zero mean vector potential for $B$ defined in terms of the Biot-Savart law $A =(-\Delta)^{-1} \nabla \times B$. Indeed, it is not hard to see that \eqref{eq:B:evo} implies that $\frac{d}{dt} {\mathcal H} = 2 \int_{\TT^d} B \cdot (u\times B) dx = 0$. This observation and the Poincar\'e inequality $\| A\|_{L^2(\TT^d)} \leq  \| B\|_{L^2(\TT^d)}$, imply the so-called {\em Arnold inequality}~\cite{Arnold74}
\begin{align}
\norm{B(\cdot,t)}_{L^2}^2 \geq  |{\mathcal H}(0)|
\,,
\label{eq:Arnold:ineq}
\end{align}
for all $t\geq 0$. Therefore, while \eqref{eq:energy} shows that the magnetic energy is strictly decreasing as long as $u\not \equiv 0$, \eqref{eq:Arnold:ineq} also shows that the magnetic energy is  bounded from below for all time, by a constant that depends only on the magnetic helicity of the initial datum.  
\end{remark}

\begin{theorem}[\bf Local existence in Sobolev spaces]
\label{thm:local}
Let $\gamma \geq 0$ and $s>d/2+1$. Assume that $B_0 \in H^s(\TT^d)$ is divergence free. Then, there exists $T_* \geq ( C \norm{B_0}_{H^s})^{-2}$, such that the active vector equation \eqref{eq:MRE} has a unique solution $B \in C^0( [0,T_*); H^s(\TT^d))$, with associated velocity $u \in C^0 ( [0,T_*); H^{s-1+2\gamma}(\TT^d)) \cap   L^2( (0,T_*); H^{s+\gamma}(\TT^d))$. Moreover,   $B$ satisfies the bound \eqref{eq:energy:a} and also
\begin{align}
\norm{B(\cdot,t)}_{\dot{H}^s}^2 \leq \norm{B_0}_{\dot{H}^s}^2 \exp\left( C \int_0^{t} \norm{\nabla u(\cdot,s)}_{L^\infty} + \norm{\nabla B(\cdot,s)}_{L^\infty}^2 ds \right)
\label{eq:blowup:criterion}
\end{align}
for $t \in [0,T_*)$, where $C>0$ is a constant which only depends on $s$, $\gamma$, and $d$. 
\end{theorem}
\begin{proof}[Proof of Theorem~\ref{thm:local}]
We use the notation $\Lambda = (-\Delta)^{1/2}$, and $\brak{A,B} = AB - BA$ for the commutator of two operators. From \eqref{eq:MRE}, we then obtain
\begin{align}
&\frac 12 \frac{d}{dt} \norm{B}_{\dot{H}^s}^2 + \norm{u}_{\dot{H}^{s+\gamma}}^2 \notag\\
&=  \int_{\TT^d} \Lambda^s u \cdot \Lambda^s (B\cdot \nabla B + \nabla p) 
+ \int_{\TT^d} \Lambda^s B \cdot \Lambda^s (B\cdot \nabla u) 
-   \int_{\TT^d} \Lambda^s B \cdot \Lambda^s (u\cdot \nabla B) 
\notag\\
&=  
\int_{\TT^d} \Lambda^s u \cdot \Lambda^s (B\cdot \nabla B)
+ \int_{\TT^d} \Lambda^s B \cdot (B\cdot \nabla \Lambda^s u)
+ \int_{\TT^d} \Lambda^s B \cdot \brak{\Lambda^s, B\cdot \nabla} u   
-   \int_{\TT^d} \Lambda^s B \cdot \brak{\Lambda^s, u\cdot \nabla} B
 \notag\\
&=  
\int_{\TT^d} \Lambda^s u \cdot \brak{\Lambda^s , B\cdot \nabla} B
+\int_{\TT^d} \Lambda^s B \cdot \brak{\Lambda^s, B\cdot \nabla} u   
-   \int_{\TT^d} \Lambda^s B \cdot \brak{\Lambda^s, u\cdot \nabla} B \,.
\label{eq:Hs:energy:1}
\end{align}
Now, from~\cite[Corollary 5.2, equation (5.1)]{Li19}, by choosing $p=p_1=p_4=2$ and $p_2=p_3=\infty$, this result states that for all $s>0$ we have the following generalization of the Kato-Ponce commutator estimate:
\begin{align}
\label{eq:Kato:Ponce}
\norm{ \brak{\Lambda^s ,f} g }_{L^2} \lesssim  \norm{\Lambda^s f}_{L^2} \norm{g}_{L^\infty} + \norm{\nabla f}_{L^\infty} \norm{\Lambda^{s-1} g}_{L^2} 
\,.
\end{align}
Applying the estimate \eqref{eq:Kato:Ponce} for the pairs $(f,g) \in \{ (B, \nabla B), (B, \nabla u), (u,\nabla B) \}$, since $\brak{\nabla ,\Lambda^s} = 0$ we obtain that 
\begin{subequations}
\label{eq:Kato:Ponce:1}
\begin{align}
\norm{\brak{\Lambda^s , B\cdot \nabla} B}_{L^2} &\les  \norm{B}_{\dot{H}^s} \norm{\nabla B}_{L^\infty} \\
\norm{\brak{\Lambda^s , B\cdot \nabla} u}_{L^2} + \norm{\brak{\Lambda^s , u\cdot \nabla} B}_{L^2} &\les  \norm{B}_{\dot{H}^s} \norm{\nabla u}_{L^\infty} + \norm{\nabla B}_{L^\infty} \norm{u}_{\dot{H}^s} 
\,.
\end{align}
\end{subequations} 
By combining \eqref{eq:Hs:energy:1} and \eqref{eq:Kato:Ponce:1}, we arrive at 
\begin{align}
\frac 12 \frac{d}{dt} \norm{B}_{\dot{H}^s}^2 + \norm{u}_{\dot{H}^{s+\gamma}}^2 
\lesssim  \norm{u}_{\dot{H}^s} \norm{B}_{\dot{H}^s} \norm{\nabla B}_{L^\infty} + \norm{B}_{\dot{H}^s}^2 \norm{\nabla u}_{L^\infty}
\,.
\label{eq:Hs:energy:2}
\end{align}
Since $\gamma \geq 0$ and $u$ has zero mean on $\TT^d$ we have that $\norm{u}_{\dot{H}^s} \lesssim \norm{u}_{\dot{H}^{s+\gamma}}$, while the condition $s>d/2+1$ implies that $\norm{\nabla u}_{L^\infty} \lesssim \norm{u}_{\dot{H}^s}$ and $\norm{\nabla B}_{L^\infty} \lesssim  \norm{B}_{\dot{H}^s}$. Thus, estimate \eqref{eq:Hs:energy:2} readily implies that there exists a constant $C = C(\gamma,s,d)> 0$ such that 
\begin{align}
\frac{d}{dt} \norm{B}_{\dot{H}^s}^2 + \norm{u}_{\dot{H}^{s+\gamma}}^2 
\leq C \norm{B}_{\dot{H}^s}^4
\,.
\label{eq:Hs:energy:3}
\end{align}
From the a-priori estimates \eqref{eq:energy} and \eqref{eq:Hs:energy:3}, the local existence of $C^0_t H^s_x$ solutions of \eqref{eq:MRE} readily follows from a standard approximation procedure, and the local time of existence is at least as large as $( C \norm{B_0}_{H^s})^{-2}$. Note that since $H^{s-1}$ is an algebra, we immediately obtain from \eqref{eq:u:def} that $u \in C^0_t H^{s-1+2 \gamma}_x$, while from \eqref{eq:Hs:energy:3} we obtain that $u \in L^2_t H^{s+\gamma}_x$. Interestingly, when $\gamma \geq 1$, the former information (the uniform in time one) provides more regularity in space than the latter one (the integrated in time one). The bound \eqref{eq:blowup:criterion} is an immediate consequence of \eqref{eq:Hs:energy:2}, since $s>d/2+1$ and $\gamma \geq 0$. 
\end{proof}

\section{Global existence  for   $\gamma > d/2 + 1 $}
\label{sec:global:Sobolev}

\begin{theorem}[\bf Global existence for the strongly regularized system]
\label{thm:global:regularized}
Let $\gamma,s > d/2+1$. Assume that $B_0 \in H^s(\TT^d)$ is divergence free. Then, the local in time solution established in Theorem~\ref{thm:local} is in fact global in time, meaning that $T_* = + \infty$, and we have that
\begin{align}
\norm{B(\cdot,t)}_{\dot{H}^s}^2 
&\leq \norm{B_0}_{\dot{H}^s}^2  \exp\left(C  t^{1/2} \norm{B_0}_{L^2} \right) \notag\\
&\quad \times
\exp\left(C  t \left( \norm{\nabla B_0}_{L^\infty}^2 + C  t^2  \norm{B_0}_{L^\infty}^6 \right) \exp\left(C  t^{1/2} \norm{B_0}_{L^2} \right)  \right)
\label{eq:global:Hs} 
\end{align}
for all $t\geq 0$, where $C=C(\gamma,s,d)>0$ is a constant.
\end{theorem}
\begin{proof}[Proof of Theorem~\ref{thm:global:regularized}]
Estimate~\eqref{eq:blowup:criterion} shows that the local in time $H^s$ solution may be uniquely continued past $T$ if $ \int_0^{T} \norm{\nabla u(\cdot,s)}_{L^\infty} + \norm{\nabla B(\cdot,s)}_{L^\infty}^2 ds < \infty$.
Thus, the global existence of smooth solutions is established if we show that the Lipschitz norm of $u$ is integrable in time, and that the Lipschitz norm of $B$ is square integrable in time.

The condition $\gamma > 1 + d/2$ implies by the Sobolev embedding that $H^\gamma \subset {\rm Lip}$. Thus, from \eqref{eq:energy:a} we deduce   
\begin{align}
\int_0^t \norm{\nabla u(\cdot,s)}_{L^\infty} ds 
\lesssim   \int_0^t \norm{u(\cdot,s)}_{\dot{H}^\gamma} ds 
\lesssim    t^{1/2} \norm{B_0}_{L^2}  \,.
\label{eq:u:Lip:good}
\end{align}
Once $u$ satisfies \eqref{eq:u:Lip:good},  we may use the following classical fact: the solution $B$ of \eqref{eq:B:evo} is given by the vector transport formula
\begin{align}
B(X(\alpha,t),t) = \nabla_\alpha X(\alpha,t) B_0(\alpha) 
\label{eq:vector:transport}
\end{align}
where $X(\alpha,t)$ is the $\TT^d$-periodic flow of the vector field $u$; that is, the solution of the ODEs
\begin{align}
\frac{d}{dt} X(\alpha,t) = u (X(\alpha,t),t), \qquad X(\alpha,0) = \alpha\,
\label{eq:Lagrangian:ODEs}
\end{align}
and $\alpha \in \TT^d$ denotes a {\em Lagrangian  label}. Differentiating \eqref{eq:Lagrangian:ODEs} with respect to $\alpha$ and appealing to \eqref{eq:u:Lip:good}, we deduce that 
\begin{align}
\norm{\nabla X(\cdot,t)}_{L^\infty} 
\leq \exp\left( \int_0^t \norm{\nabla u(\cdot,s)}_{L^\infty} ds \right) 
&\leq \exp\left( C t^{1/2} \norm{B_0}_{L^2}  \right) 
\,.
\label{eq:nabla:X:bdd}
\end{align}
Thus, upon composing \eqref{eq:vector:transport} with the back-to-labels map $X^{-1}(x,t)$, and appealing to \eqref{eq:nabla:X:bdd}, we obtain that 
\begin{align}
\norm{B(\cdot,t)}_{L^\infty} \leq \norm{B_0}_{L^\infty}  \exp\left( C  t^{1/2} \norm{B_0}_{L^2}  \right) 
\label{eq:B:bounded}
\end{align}
for all $t>0$. 

It remains to estimate the $L^\infty$ norm of  $\nabla B$. For this purpose we differentiate \eqref{eq:B:evo} with respect to $x$, and contract the resulting equation with $\nabla B$ to deduce
\begin{align*}
(\partial_t + u\cdot\nabla)  |\nabla B|^2 \leq  4 |\nabla u| |\nabla B|^2 + 2 |\nabla^2 u| |B| |\nabla B| \,.
\end{align*}
By the maximum principle, we obtain that
\begin{align}
\norm{\nabla B(\cdot,t)}_{L^\infty} 
&\leq \norm{\nabla B_0}_{L^\infty}  \exp\left(2 \int_0^t \norm{\nabla u(\cdot,s)}_{L^\infty} ds \right) \notag\\
&\qquad + \int_0^t \norm{\nabla^2 u(\cdot,s)}_{L^\infty} \norm{B(\cdot,s)}_{L^\infty}   \exp\left(2 \int_s^t \norm{\nabla u(\cdot,s')}_{L^\infty} ds' \right) ds
\label{eq:nablaB:ode}
\,.
\end{align}
Thus, we need a bound on the $L^\infty$ norm of the Hessian of $u$. For this purpose we note that the condition $\gamma > d/2+1$ trivially implies that $\gamma > 3/2$, and thus the bound
\begin{align*}
\norm{\nabla^2 (-\Delta)^{- \gamma} \Proj \div \varphi}_{L^\infty} &\lesssim  \norm{\varphi}_{L^\infty}  
\end{align*}
holds for every integrable  $\TT^d$-periodic $2$-tensor $\varphi$. Combining the above estimate with the constitutive law \eqref{eq:B:to:u}, we obtain
\begin{align*}
\norm{\nabla^2 u(\cdot,t)}_{L^\infty} 
\lesssim  \norm{B(\cdot,t)\otimes B(\cdot,t)}_{L^\infty} 
\lesssim  \norm{B(\cdot,t)}_{L^\infty}^2 
\,,
\end{align*}
for all $t>0$. From the above display, \eqref{eq:u:Lip:good}, \eqref{eq:B:bounded}, and \eqref{eq:nablaB:ode} we deduce
\begin{align}
\norm{\nabla B(\cdot,t)}_{L^\infty} 
&\leq \left( \norm{\nabla B_0}_{L^\infty} + C  t  \norm{B_0}_{L^\infty}^3 \right) \exp\left(C  t^{1/2} \norm{B_0}_{L^2} \right) 
\label{eq:nablaB:bound}
\,,
\end{align}
for all $t>0$, where $C>0$ is a sufficiently large constant which depends on $\gamma$ and $d$.

The bounds \eqref{eq:u:Lip:good} and \eqref{eq:nablaB:bound} conclude the proof of global existence of $H^s$ solutions. Taking into account the estimate \eqref{eq:blowup:criterion} we also obtain the bound \eqref{eq:global:Hs}.
\end{proof}

\begin{remark}
The condition $\gamma > d/2+1$ implies that   for $0 < \epsilon < 2 (\gamma - d/2-1)$ we have the bound
\begin{align}
\norm{\nabla (-\Delta)^{- \gamma} \Proj \div \varphi}_{C^\epsilon} &\lesssim \norm{\varphi}_{L^1}  
\label{eq:Sobolev:again}
\end{align}
holds for every integrable  $\TT^d$-periodic $2$-tensor $\varphi$. Combining \eqref{eq:Sobolev:again}, \eqref{eq:B:to:u}, and \eqref{eq:energy:a} we first deduce that 
\begin{align}
\norm{\nabla u(\cdot,t)}_{C^\epsilon} 
\lesssim \norm{B(\cdot,t)\otimes B(\cdot,t)}_{L^1} 
\lesssim   \norm{B(\cdot,t)}_{L^2}^2 
\lesssim  \norm{B_0}_{L^2}^2 \,,
\label{eq:u:Lip}
\end{align}
for all $t>0$. We note that this bound is pointwise in time, in contrast to \eqref{eq:u:Lip:good} which is time integrated. 
\end{remark}

\section{Convergence as $t\to \infty$   for   $\gamma > d/2 + 1 $}
\label{sec:t:to:infinity}
In view of Theorem~\ref{thm:global:regularized}, we know that if the initial datum lies in $H^s(\TT^d)$ and the regularization parameter $\gamma$ in \eqref{eq:B:to:u} is sufficiently large, namely $\gamma> d/2+1$, then the system \eqref{eq:MRE} has global existence of solutions. In this section we discuss the possible behavior of these solutions as $t\to \infty$. 

Our first result shows that as $t\to \infty$, the velocity field $u(\cdot,t)$ converges to $0$. 
\begin{theorem}[\bf Asymptotic behavior for the velocity]
\label{thm:velocity:relaxation}
Let $\gamma,s > d/2+1$ and assume that $B_0 \in H^s(\TT^d)$ is divergence free. 
Then the zero mean velocity field $u$ associated to the magnetic field $B \in C^0([0,\infty);H^s(\TT^d))$ has the property that 
\begin{align}
\lim_{t\to \infty} \norm{\nabla u (\cdot,t)}_{L^\infty} = 0 \,.
\label{eq:velocity:relaxation}
\end{align}
\end{theorem}
\begin{proof}[Proof of Theorem~\ref{thm:velocity:relaxation}]
The proof is based on the bound  \eqref{eq:u:Lip}, on the energy inequality \eqref{eq:energy:a}, and on a bound for the time derivative of $u$, which we claim satisfies
\begin{align}
 \norm{\partial_t u(\cdot,t)}_{C^\epsilon}  \lesssim \norm{B_0}_{L^2}^4
 \label{eq:dt:u:Holder}
\end{align}
for all $t\geq 0$. In order to prove \eqref{eq:dt:u:Holder} we apply a   time derivative to \eqref{eq:u:def}, and replace $\partial_t B$ in the resulting formula via \eqref{eq:B:evo}, to arrive at
\begin{align*}
(-\Delta)^{\gamma} \p_t u_i - \p_i (\p_t p)
&= \p_t (B_j \p_j B_i) \notag \\
&= \p_j \p_t (B_j B_i) \notag \\
&= \p_j (\p_t B_j B_i + B_j \p_t B_i) \notag \\
&= \p_j \left( (B_k \p_k u_j - u_k \p_k B_j)B_i + B_j (B_k \p_k u_i - u_k \p_k B_i) \right) \notag \\
&= \p_j \left( B_i B_k \p_k u_j + B_j  B_k \p_k u_i - u_k B_i \p_k B_j   - u_k B_j \p_k B_i \right) \notag \\
%&= \p_j \left( B_i B_k \p_k u_j + B_j  B_k \p_k u_i\right)  - \p_j \left( u_k B_i \p_k B_j   + u_k B_j \p_k B_i \right) \notag \\
&= \p_j \left( B_i B_k \p_k u_j + B_j  B_k \p_k u_i\right)  - \p_j \left( u_k \p_k (B_i  B_j) \right)\notag \\
&= \p_j \left( B_i B_k \p_k u_j + B_j  B_k \p_k u_i\right)  - \p_j \p_k \left( u_k  B_i  B_j \right) 
\end{align*}
for every component $i \in \{1,\ldots,d\}$. Therefore,  we have established that
\begin{align*}
 \p_t u  
&= (-\Delta)^{- \gamma} \Proj \div \left( B \otimes (B\cdot \nabla) u + (B\cdot \nabla) u \otimes B \right)  - (-\Delta)^{-\gamma} \Proj \div \div \left(B\otimes B \otimes u \right) \,.
\end{align*}
Since $\gamma > 1 + d/2$, we may again use inequality \eqref{eq:Sobolev:again} along with the Poincar\'e inequality,  and deduce from 
the above   formula for $\partial_t u$ that 
\begin{align*}
\norm{\p_t u (\cdot,t)}_{C^\epsilon} 
&\lesssim  \norm{(B \otimes B \otimes \nabla u)(\cdot,t)}_{L^1} 
+ \norm{(B \otimes B \otimes   u)(\cdot,t)}_{L^1} \notag\\
&\lesssim  \norm{B(\cdot ,t)}_{L^2(\TT^d)}^2 \norm{u(\cdot,t)}_{W^{1,\infty}} \notag\\
%&\lesssim_{\gamma,\epsilon} \norm{B_0}_{L^2(\TT^d)}^2 \norm{u(\cdot,t)}_{W^{1,\infty}(\TT^d)}\notag\\
&\lesssim  \norm{B_0}_{L^2(\TT^d)}^4\,.
\end{align*}
In the last inequality above we have appealed to \eqref{eq:energy:a} and \eqref{eq:u:Lip}.
Thus, we have shown that \eqref{eq:dt:u:Holder} holds.

In order to conclude the proof, we note that the energy inequality \eqref{eq:energy:a} and the Sobolev embedding $H^\gamma \subset L^\infty$ gives that 
\begin{align*}
\int_0^\infty \norm{u(\cdot,t)}_{L^\infty}^2 &\lesssim  \norm{B_0}_{L^2(\TT^d)}^2 \,.
\end{align*}
Combined with \eqref{eq:dt:u:Holder}, the above estimate shows that 
\begin{align*}
\lim_{t\to\infty} \norm{u(\cdot,t)}_{L^\infty} = 0 \,.
\end{align*}
The conclusion \eqref{eq:velocity:relaxation} now follows by interpolating the $L^\infty$ norm of $\nabla u$ between the $L^\infty$ norm of $u$, which vanishes as $t\to \infty$ as shown above, and the $C^\epsilon$ norm of $\nabla u$, which is uniformly bounded  by \eqref{eq:u:Lip}. 
\end{proof} 

\begin{remark}[\bf Relaxation towards Euler steady states?]
\label{rem:relaxation?}
Since \eqref{eq:velocity:relaxation} shows that $\norm{\nabla u(\cdot,t)}_{L^\infty}$ vanishes as $t\to \infty$, it is tempting to conjecture (as was already done by Moffatt~\cite{Moffatt85}) that as $t\to \infty$ the magnetic field $B(\cdot,t)$ relaxes to a steady state $\bar B$ which solves \eqref{eq:u:def} with the left hand side equals to zero; that is, $\bar B$ is a stationary solution of the incompressible Euler equations. The purpose of this remark is to argue that the information provided in Theorem~\ref{thm:velocity:relaxation} does not appear to be sufficient to conclude this statement. By \eqref{eq:energy:a} and weak compactness we do have the existence of subsequences $t_k \to \infty$ such that the associated magnetic fields $B_k(x) = B(x,t_k)$ converge weakly in $L^2$; say $B_k \rightharpoonup \bar B$, for some (weakly) incompressible vector field $\bar B$. Additionally, \eqref{eq:energy} shows that the sequence $\{\|  B_k\|_{L^2} \}_{k\geq 1}$ is strictly decreasing and non-negative, so that there exists $\bar E \geq 0$ with $\lim_{k\to \infty} \|B_k\|_{L^2} = \bar E$; in fact $\bar E>0$ in view of the Arnold inequality \eqref{eq:Arnold:ineq} as long as the initial datum is topologically nontrivial. Of course, we do not know whether $\|\bar B\|_{L^2}$ equals to $\bar E$, or else we'd have strong $L^2$ convergence as $k \to \infty$. The additional information provided by Theorem~\ref{thm:velocity:relaxation} and \eqref{eq:u:def} gives that $\Proj \div (B_k \otimes B_k) \rightharpoonup 0$. These facts do not, however, seem to imply that $\Proj\div(\bar B \otimes \bar B) = 0$ in the sense of distributions, which would be the condition that $\bar B$ is a stationary Euler flow. Indeed, we may only deduce that 
 \begin{align*}
     \Proj\div(\bar{B} \otimes \bar{B}) 
     %&= \div( (\bar{B} - B_k ) \otimes \bar{B} ) + \div( B_k \otimes \bar{B}) \\
%     &= \div( (\bar{B} - B_k ) \otimes \bar{B} ) + \div(B_k \otimes B_k) +  \div (B_k \otimes (\bar{B} - B_k)) \\
     &= \Proj\div((\bar{B} - B_k) \otimes \bar{B}) + \Proj\div (\bar{B} \otimes (\bar{B} - B_k)) \notag\\
     &\qquad + \Proj\div(B_k \otimes B_k) - \Proj\div((B_k - \bar{B}) \otimes (B_k  - \bar{B}  ) )
     \,.
 \end{align*}
 The first three terms on the right side of the above do converge to $0$ weakly as $k\to \infty$ (the first two terms by the assumption that $B_k \rightharpoonup \bar B$, and the third term due to Theorem~\ref{thm:velocity:relaxation}). However, we do not have enough information to conclude that the fourth term converges to $0$ as $k\to \infty$; the information that $(B_k - \bar{B}) \otimes (B_k  - \bar{B}  )$ is a symmetric non-negative tensor, which is uniformly bounded in $L^1$, is not sufficient; the enemy is $\|\bar B\|_{L^2} < \bar E$.
 \end{remark}

\begin{remark}[\bf Time integrability of $\norm{\nabla u(\cdot,t)}_{L^\infty}$]
\label{rem:good:bounds?}
From \eqref{eq:energy:a} and the Sobolev embedding $H^\gamma \subset W^{1,\infty}$ we may deduce that $\norm{\nabla u(\cdot,t)}_{L^\infty} \in L^2(0,\infty)$. It remains however an open problem to show that the Lipschitz norm of $u$ decays sufficiently fast to ensure that $\norm{\nabla u(\cdot,t)}_{L^\infty} \in L^1(0,\infty)$. If this faster decay were true, then in view of the $\dot{H}^1$ energy estimate for \eqref{eq:MRE}, which after exploring a few cancelations can be shown to be
\begin{align*}
\frac 12 \frac{d}{dt} \norm{\nabla B}_{L^2}^2 + \norm{u}_{\dot{H}^{\gamma+1}}^2 \leq 3  \norm{\nabla B}_{L^2}^2 \norm{\nabla u}_{L^\infty} \,,
\end{align*}
would imply that $\norm{\nabla B(\cdot,t)}_{L^2}$ is uniformly bounded in time. In turn, such information would be sufficient to extract as $t\to \infty$ limit points $\bar B$, which are stationary solutions of the Euler equations. We note, however, that at least when $\gamma =0$, in Theorem~\ref{thm:growth} we show, for suitable choices of initial data, the $\dot{H}^1$ norm of $B$ does not remain uniformly bounded in time. Thus it is not possible for the Lipschitz norm of $u$ to be integrable in time. Whether this situation is {\em generic} remains an open problem.
\end{remark}

\section{2D stability of the state $B = e_1$ and $u= 0$}
\label{sec:2D:stability}
We consider the MRE system \eqref{eq:MRE} in two space dimensions with $\gamma=0$.  In this section we study the asymptotic stability of the steady state 
$$
B = e_1
\qquad \mbox{and} \qquad 
u =0 \,,
$$ under Sobolev smooth perturbations. We note that for the  MHD system {\em with viscosity but no resistivity} on $\RR^2$, i.e. for \eqref{eq:MHD} with $\Delta u$ added to \eqref{eq:MHD:b}, the stability of $(u ,B) = (0,e_1)$ was proved for Sobolev smooth perturbations with certain admissibility conditions for the initial data of magnetic perturbations in \cite{LinXuZhang15} (see also  \cite{RenWuXiangZhang14} where the admissibility conditions were removed). These works make use of the fact that at the linearized level $u$ satisfies
$$
\p_{t}^2 u - \Delta \p_t u - \p_1^2 u = 0 \,.
$$
For the magnetic relaxation equations \eqref{eq:MRE}, this favorable structure is no longer available since $u$ is completely determined from $B$ through  \eqref{eq:u:def}. Linearizing around $(u ,B) = (0,e_1)$ instead leads to the partially dissipative equation
\begin{equation}
\label{eq:partial:dissipation}
\p_t B = \p_1^2 B \,.
\end{equation}
This motivates the use of a different approach to proving global existence for the perturbations, as in~\cite{Elgindi17}.

The perturbation of the magnetic field around the steady state $e_1$ is written as $b$, i.e. we consider 
\begin{align}
b:= B -  e_1 
\,.
\label{eq:b:def}
\end{align}
From \eqref{eq:u:def}, with $\gamma = 0$, we deduce that 
\begin{align}
\label{eq:velocity}
u = \p_1 b  + v
\end{align}
where the nonlinear part of the velocity, denoted as $v$, is given by 
\begin{align}
\label{def:perturbation:velocity}
    v := b \cdot \nabla b + \nabla p \,, \qquad \div v = 0 \,.
\end{align}
Inserting the ansatz \eqref{eq:b:def}--\eqref{def:perturbation:velocity} into \eqref{eq:MRE} we obtain the evolution equation for the perturbation of the magnetic field
\begin{align*}
    \partial_t b + v \cdot \nabla b - b \cdot \nabla v - \partial_1^2 b = b \cdot \nabla \partial_1 b- \partial_1 b \cdot \nabla b + \partial_1 v
    \, .
\end{align*}
Using \eqref{def:perturbation:velocity} we arrive at the following system for the magnetic perturbation:
\begin{subequations}
     \label{eq:magnetic:perturbation2}
 \begin{align}
     \partial_t b + v \cdot \nabla b - b \cdot \nabla v - \partial_1^2 b &= 2 \Proj ( b \cdot \nabla \partial_1 b) \label{eq:magnetic:perturbation2a}\\
      v &= b \cdot \nabla b + \nabla p  
      \label{eq:magnetic:perturbation2b} \\
     \div v &= \div b = 0 \,.
 \end{align}
 \end{subequations}
Before stating our main theorem, it will be useful to introduce notation for the $x_1$-independent and the $x_1$-dependent components of $b$. As such, for any function $\psi \colon \TT^2 \to \RR$, we define
\begin{align*}
    \po \psi (x_2) &:=  \fint_{\TT} \psi (x_1, x_2 )  dx_1 \\
    \pn \psi (x_1,x_2) &:= \psi (x_1,x_2)- \po \psi (x_2) \,. 
\end{align*}
With this notation, our main result concerning the system~\eqref{eq:magnetic:perturbation2} is:

\begin{theorem}[\bf Stability and relaxation]
\label{theorem:main}
Let $k\geq 4$ and $m \geq k  +9$. Choose $\delta \in (0,1)$. There exists $\varepsilon_0$ such that if 
\begin{align}
\| b_0\|_{H^m} = \varepsilon \leq \varepsilon_0 \,,
\label{eq:b0:Hm}
\end{align}
and $\po b_{0} = 0$, then we have that \eqref{eq:magnetic:perturbation2} has a unique global in time smooth solution $(b,v)$, which satisfies $\norm{b(\cdot,t)}_{L^2} \leq \eps$ and 
\begin{subequations}
\label{eq:thm:stability}
\begin{align}
     &\norm{\pn b(\cdot,t)}_{\dot{H}^k} \leq 4 \varepsilon e^{-(1 - \delta)t}  \label{eq:f:Hk}  \\
     &\norm{\po b_1(\cdot,t)}_{H^{k  +2 }} \leq 4  \varepsilon  \label{eq:a:H:k+2} \\
     &\norm{b(\cdot,t)}_{\dot{H}^m}^2  \leq  4 \eps e^{  \varepsilon t } \label{eq:b:Hm}
 \end{align}
 \end{subequations}
for $t \in [0, \infty)$. As a consequence, the total velocity field   satisfies $u(\cdot,t) \to 0$ as $t\to \infty$, whereas the total magnetic field $B (\cdot,t)= e_1 + b(\cdot,t)$ relaxes to a steady state $\bar B$ with $\|\bar B - e_1\|_{H^{k+2}} \leq 4 \eps$, both convergences taking place  with respect to strong topologies.
\end{theorem}

\begin{remark}[\bf Notation]
 For simplicity of notation, throughout the proof of Theorem~\ref{theorem:main}, we shall use the notation:
\begin{subequations}
\label{eq:a:f:w:def}
\begin{align}
a &= a(x_2,t) = \po b_1 (x_2,t) \\
f &= f(x_1,x_2,t) = \pn b (x_1,x_2,t) \\
w &= w(x_1,x_2,t) = \pn v(x_1,x_2,t) \,.
\end{align}
\end{subequations}
We do not introduce new notation for $\po v_1$, but note from \eqref{eq:magnetic:perturbation2b} and the observation that $\po b_2 (\cdot,t)= 0$ (which will be established in Lemma~\ref{lem:po:b:evo} below), we have
\begin{align*}
\po v_1 = \p_2 \po (b_2 b_1) =  \p_2 \po (f_2 (f_1 + a)) =  \p_2 \po (f_2  f_1)
\,.
\end{align*}
The above identity will be used in the analysis below. Note that with the notation in \eqref{eq:a:f:w:def}, we have that the stability estimates in \eqref{eq:thm:stability} become
\begin{align}
\norm{f(\cdot,t)}_{\dot{H}^k} \leq 4 \eps e^{-(1-\delta)t}\,,
\qquad 
\norm{a(\cdot,t)}_{H^{k+2}} \leq 4 \eps \,,
\qquad
\norm{b(\cdot,t)}_{\dot{H}^m} \leq  4 \eps e^{  \varepsilon t } 
\,,
\label{eq:thm:stability:alt}
\end{align}
which are the bounds proven below. 
\end{remark}

\subsection{The evolution equations for $a$ and $f$}
Before turning to the proof of Theorem~\ref{theorem:main}, we need to determine the evolution equations for $a$ and $f$. In turn, this is necessary because the $\p_1^2 b$ dissipative term present in \eqref{eq:magnetic:perturbation2a} may only be expected to cause decay of the part of $b$ which is not constant in the $x_1$ direction, i.e. for $f$. Moreover, the precise coupling between the evolution equations for $a$ and $f$ is crucial to the proof (and is also the reason why in three dimensions this stability result doesn't hold). In this direction, for $a$ we have:

\begin{lemma}[\bf The $a$ evolution]
\label{lem:po:b:evo}
 Assume that $(b,v)$ are smooth solutions of \eqref{eq:magnetic:perturbation2} and $\po b_{0} = 0$. Then, we have 
\begin{align}    
    \partial_t a  
    &= \partial_2\po ( 2 f_2 \partial_1 f_1  + f_2 w_1 - w_2 f_1 ) 
    \label{eq:b:zero:mode}
    =: N'(f,w)
\end{align}
and $\po b_2 (\cdot,t) = 0$. Crucially, $a$ does not appear on the right side of \eqref{eq:b:zero:mode}.
\end{lemma}
\begin{proof}[Proof of Lemma~\ref{lem:po:b:evo}]
Applying $\po$ to \eqref{eq:magnetic:perturbation2a} and using that $\po \p_1 \psi = 0$ for any periodic $\psi$, gives
\begin{equation}
    \partial_t \po b_i = \partial_2 \po(  b_2 v_i - v_2 b_i  + 2b_2 \partial_1 b_i).
    \label{eq:temp:1}
\end{equation}

When $i=2$, we appeal to the fact that $\po (b_2 \p_1 b_2) = \po \p_1 (b_2^2/2) = 0$, and to the assumption $\po b_{2,0} = 0$, to conclude from \eqref{eq:temp:1} that $\po b_2 (x_2,t) = 0$ for all $t\geq 0$. In particular, this implies that $f_2 = b_2$. Moreover, since $\fint_{\TT^2} f = 0$ and $0 = \div b = \div f$, there exists a periodic stream function $\phi$, such that $f = \nabla^\perp \phi$. In particular, 
$$
f_2 = \p_1 \phi \,.
$$
Similarly, since $\div v = 0$ and $\fint_{\TT^2} v = \fint_{\TT^2} \div ( b\otimes b + p I_2)  = 0$, we obtain that there exists a periodic stream function $\varphi$, such that $v = \nabla^\perp \varphi$. 
In particular,
$$
v_2 = w_2= \p_1 \varphi\,.
$$

With this information, we return to \eqref{eq:temp:1} and set $i=1$. Since $\partial_1 b_1 = \partial_1 f_1$, $b_2 = f_2 $, and $v_2 =w_2$, we have that 
\begin{align*}
\po(b_2 v_1 - v_2 b_1  + 2b_2 \partial_1 b_1) 
&=  \po(f_2 v_1 - w_2 b_1  + 2 f_2 \partial_1 f_1) \\
&=  \po(f_2 w_1 - w_2 f_1  + 2 f_2 \partial_1 f_1) 
+ \po (\pn b_2 \po v_1 - \pn v_2 \po b_1) 
\end{align*}
which establishes \eqref{eq:b:zero:mode}, upon noting that $\po (\pn \psi_1 \po \psi_2) = \po \psi_2 \po (\pn \psi) = 0$ for any   $\psi_1 , \psi_2$.
\end{proof}

The evolution equation for $f$ is more complicated, and is given by the following lemma. 

\begin{lemma}[\bf The $f$ evolution]
\label{lem:pn:b:evo}
Assume that $(b,v)$ are smooth solutions of \eqref{eq:magnetic:perturbation2} and $\po b_{0} = 0$. Then, we have 
  \begin{equation}
\label{eq:full:equation}
    \partial_t f = L (f) + N(f,w)
\end{equation}
where the linear operator $L$ acts on the vector field $f =(f_1, f_2)^T$ as 
\begin{subequations}
\begin{align}
    L (f) &:= ( 1 + a)^2\partial_1^2 f + ( 1 + a )\nabla \partial_1 p_L - \partial_2 a \partial_2 p_L e_1
    \label{eq:linear:operator:vector:form} \\
    p_L &:= p_L(a, f) = 2 (-\Delta)^{-1}(\partial_2 a \partial_1 f_2)
    \label{def:lin:pressure}
\end{align}
\end{subequations}
where as in \eqref{eq:a:f:w:def}, $a = \po b_1$. 
The nonlinear operator $N$ appearing in \eqref{eq:full:equation} is defined as
\begin{subequations}
\begin{align}
    N(f,w) 
    &:= a \partial_1 \pn (f \cdot \nabla f + \nabla p_N) + \pn(f \cdot \nabla w - w \cdot \nabla f ) - \partial_2 \po(f_1 f_2)\partial_1 f \notag\\
    &\quad + 2\pn(f \cdot \nabla \partial_1 f) + \nabla \partial_1 \pn p_N + \left( \partial_2^2 \po(f_1 f_2)f_2 - \partial_2 a \pn(f \cdot \nabla f_2 + \partial_2 p_N) \right) e_1
    \label{def:nonlinearity}
\\
       p_N &:= p_N(f) = 2 (-\Delta)^{-1} ((\partial_1 f_1)^2 + \partial_1 f_2 \partial_2 f_1)).
        \label{def:nonlin:pressure}
\end{align}
\end{subequations}
\end{lemma}
\begin{proof}[Proof of Lemma~\ref{lem:pn:b:evo}]
We apply $\pn$ to \eqref{eq:magnetic:perturbation2a} to get 
\begin{equation}
    \partial_t f - \partial_1^2 f = \pn( b \cdot \nabla v - v \cdot \nabla b + 2 b \cdot \nabla \partial_1 b + \nabla \partial_1 p )
    \,.
    \label{eq:pt:f:1}
\end{equation}
The goal is to further decompose the right side of \eqref{eq:pt:f:1}, in order to extract from it all local and nonlocal terms which are linear in $f$.

We first determine a decomposition for the pressure. Applying a divergence to \eqref{def:perturbation:velocity} gives
\begin{align}
\label{eq:pressure}
-\Delta p 
= \div (b \cdot \nabla b) 
&= 2 (\partial_1 b_1)^2 + 2 \partial_1 b_2 \partial_2 b_1 \notag\\
&= \underbrace{2( (\partial_1 f_1)^2 + \partial_1 f_2 \partial_2 f_1)}_{=: - \Delta p_N} +  \underbrace{2  \partial_1 f_2 \partial_2 a}_{=: -\Delta p_L} 
\end{align}
where $p_N$ is the pressure which is nonlinear in $f$, and $p_L$ is the pressure which is linear with respect to $f$. Note that both of these pressure terms are uniquely defined once we impose that they have zero mean on $\TT^2$, and that they correspond to definitions \eqref{def:lin:pressure} and \eqref{def:nonlin:pressure}.
 
Next, we compute the velocity in terms of the magnetic perturbation. As noted in Remark~\ref{eq:a:f:w:def}, we may decompose the velocity field as 
 \begin{align}
 \label{eq:velocity:decomposition}
     v_1  = w_1 + \partial_2 \po (f_1 f_2)\,,
     \qquad 
     v_2  = w_2 \, .
 \end{align}
Furthermore, by applying $\pn$ to \eqref{def:perturbation:velocity}, and using \eqref{eq:pressure},  we obtain that 
\begin{align}
    \label{def:u:bar}
    w
    &=\pn ( b_1 \partial_1 b + b_2 \partial_2 b + \nabla p) \notag \\ 
    &= a \partial_1f + \nabla p_L + 
    \partial_2 a f_2 e_1
+  \pn(f \cdot \nabla f + \nabla p_N) \,.
   \end{align}
   With \eqref{def:u:bar} in hand, we now compute the stretching and the advection terms present on the right side of \eqref{eq:pt:f:1}. Indeed, from \eqref{eq:velocity:decomposition} and \eqref{def:u:bar} for the stretching term in \eqref{eq:pt:f:1} we obtain 
\begin{align*}
    \pn(b \cdot \nabla v_1) 
      &= \pn \left( (a +  f_1) \partial_1w_1 + f_2 \partial_2 ( w_1 +  \partial_2 \po(f_1f_2) ) \right) \notag\\
      &=  a \partial_1 w_1 + f_2 \partial_2^2\po(f_1f_2)   +  \pn( f \cdot \nabla w_1 ) \notag \\
      &= a \partial_1( a \partial_1 f_1 + \partial_2 a f_2 + \partial_1p_L) + a \partial_1\pn(f\cdot\nabla f_1  + \partial_1 p_N )   + f_2 \partial_2^2\po(f_1f_2) +  \pn( f \cdot \nabla w_1 ) \notag \\
      &= a^2 \partial_1^2 f_1 + a \partial_2 a \partial_1f_2 + a\partial_1^2 p_L  + a \partial_1\pn(f\cdot\nabla f_1  + \partial_1 p_N )   + f_2 \partial_2^2\po(f_1f_2) +  \pn( f \cdot \nabla w_1 )
\end{align*}
and similarly, 
\begin{align*}
    \pn(b \cdot \nabla v_2) 
    &= \pn((a +  f_1) \partial_1 w_2 +f_2 \partial_2 w_2  ) \notag\\
    &=  a \partial_1w_2 + \pn(f\cdot \nabla w_2) \notag \\
    &= a \partial_1 (a \partial_1 f_2 + \partial_2 p_L) + a \partial_1 \pn( f\cdot\nabla f_2 + \partial_2 p_N)  + \pn(f \cdot \nabla w_2) \notag\\
    &= a^2 \partial_1^2f_2  + a \partial_1\partial_2  p_L + a \partial_1 \pn( f\cdot\nabla f_2 + \partial_2 p_N)  + \pn(f \cdot \nabla w_2)
    \,.
\end{align*}
On the other hand, for the transport term in \eqref{eq:pt:f:1} we have
\begin{align*}
     \pn(v \cdot \nabla b_1) &=  \pn( (w_1 + \p_2 \po(f_1 f_2))  \partial_1 f_1 + w_2 \partial_2 (f_1 + a) )\notag \\
    &= \pn(w \cdot \nabla f_1 ) + \p_2 \po(f_1 f_2) \p_1 f_1 + \p_2 a w_2 \notag \\
    &=a  \partial_2 a \p_1 f_2 + \p_2 a  \partial_2 p_L  + \p_2 a \pn( f\cdot \nabla f_2 + \p_2 p_N) + \p_2 \po(f_1 f_2) \p_1 f_1 + \pn(w \cdot \nabla f_1 )  
    \,,
\end{align*} 
and 
\begin{align*}
    \pn(v \cdot \nabla b_2) &= \pn( (w_1 + \p_2 \po(f_1 f_2))  \partial_1 f_2 + w_2 \partial_2 b_2 ) \notag \\
    &= \partial_2 \po(f_1 f_2) \partial_1 f_2 + \pn(w \cdot \nabla f_2)
    \,.
\end{align*}
For the third nonlinear term on the right side of \eqref{eq:pt:f:1} we have
\begin{align*}
 \pn(b \cdot \nabla \partial_1 b)  
   &= \pn( (a + f_1) \partial_1^2 f + f_2 \partial_2 \partial_1 f) \notag\\
   &= a \partial_1^2 f + \pn(f \cdot \nabla \p_1 f) \,.
\end{align*}
Gathering the above five displayed equations, we obtain that 
\begin{align*}
\mbox{Linear terms on right side of } \eqref{eq:pt:f:1}
&=  a^2 \p_1^2 f + a \p_1 \nabla p_L - \p_2 a \p_2 p_L e_1 + 2 a \p_1^2 f 
\,, \\
\mbox{Nonlinear terms on right side of } \eqref{eq:pt:f:1}
&= a \p_1 \pn(f\cdot \nabla f + \nabla p_N) + \pn(f\cdot \nabla w - w\cdot \nabla f)  \notag\\
&\quad  + 2 \pn(f \cdot \nabla \p_1 f) -   \p_2 \po(f_1 f_2) \p_1 f  \notag\\
&\quad    + \left( f_2 \p_2^2\po(f_1 f_2) - \p_2 a \pn(f\cdot\nabla f_2 + \p_2 p_N) \right) e_1\,.
\end{align*}
From the above displayed equations and \eqref{eq:pt:f:1}, the proof of \eqref{eq:full:equation} follows.\end{proof}

\subsection{Properties of the linear operator $L(f)$}
\label{sec:L:properties}

\begin{lemma}
\label{lem:L:properties}
Suppose $f(x_1 ,x_2)$ is sufficiently regular such that 
\begin{subequations}
\label{eq:mean:div:0}
\begin{align}
\div f &= 0 \\
\po f &=0
\end{align}
\end{subequations}
Then
\begin{subequations}
\begin{align*}
\div Lf &= 0 \\
\po (L(f)) &=0
 %\label{eq:zero:mean}
\end{align*}
\end{subequations}
\end{lemma}

\begin{proof}[Proof of Lemma~\ref{lem:L:properties}]
We can write $Lf $ as 
$$
Lf =   \partial_1 \left( (1 + a)^2  \partial_1 f +  ( 1 + a )\nabla  p_L   -2 \partial_2 a  \partial_2 (-\Delta )^{-1}(\partial_2 a f )e_1  \right)
$$
Therefore, assuming $f$ is sufficiently regular, we conclude for each  $t$ and $x_2$ we have $\po (L(f))(x_2,t) = 0$. 
Furthermore
\begin{align}
\div L f &= (1 + a)^2\partial_1^2 \div f  + (1 + a) \Delta \partial_1 p_L  + \partial_2 (1 + a)^2 \partial_1^2 f_2 - \partial_2 a \partial_1 \partial_2 p_L + \partial_2(1 + a) \partial_1 \partial_2 p_L \notag\\
           &= -2(1 + a) \partial_2 a \partial_1^2 f + 2 (1 + a)\partial_2 a \partial_1^2 f_2 \notag \\
           &= 0\,,
\end{align}
which concludes the proof.
\end{proof}
\begin{remark}[\bf Solvability of the linear equation]
Now let us consider the evolution equation
\begin{equation}
\label{eq:linear:evo}
\partial_t f = Lf \,, \qquad f|_{t=0} = f_0 \,,
\end{equation}
where the initial data $f_0$ satisfies \eqref{eq:mean:div:0}, i.e. it is divergence free and its zero frequency in the $x_1$ variable is trivial. Using Lemma \ref{lem:L:properties} and the energy estimates done in Proposition \ref{prop:linear:decay} we can show that for sufficiently regular initial data $f_0$, the unique solution $f$ of \eqref{eq:linear:evo} also satisfies \eqref{eq:mean:div:0}.
\end{remark}

Before we state our main semigroup estimate, Proposition~\ref{prop:linear:decay} below, we specify the function spaces where we consider the evolution of solutions of \eqref{eq:linear:evo}.
For $k \in \mathbb{N}$ we define
\begin{equation}
\label{eq:zero:mean:Sobolev}
\dot{H}_0^k:= \left\{ f \in H^k(\mathbb{T}^2 ; \mathbb{R}^2) \colon \po f_j  = 0,  j \in \{ 1,2 \}  \right\} .
\end{equation}
Note  that $\dot{H}_0^k$ embeds into $\dot{H}_0^l$ for any $l \leq k$ because Poincare's inequality holds. 
\begin{prop}[\bf Linear decay estimates]
\label{prop:linear:decay}
Let $f$ be a solution of \eqref{eq:linear:evo}.
For any $\delta \in (0,1)$ there exists $\varepsilon_0 > 0$ such that for any $0 < \varepsilon \leq  \varepsilon_0$
if
\begin{equation}
    \|a(\cdot, t) \|_{W^{k + 1, \infty}} \leq 4 \varepsilon  \quad t \in [0 ,T]
    \label{eq:Lagavulin}
 \end{equation}
then 
\begin{equation}
    \label{eq:semigroup}
    \| e^{L t }\|_{\dot{H}_0^k \to \dot{H}_0^k} \leq e^{-(1 - \delta) t} \quad t \in [0, T]
\end{equation}
where $k$ is as in Theorem \ref{theorem:main}. 
\end{prop}
\begin{proof}[Proof of Proposition~\ref{prop:linear:decay}]
Differentiating  \eqref{eq:linear:operator:vector:form}  $k$ times with respect to $\partial_1$, multiplying by $\partial_1^k f$, and then integrating gives
\begin{align}
      \frac{1}{2} \frac{d}{dt} \| \partial_1^k f\|_{L^2}^2 + \| (1 + a) \partial_1^{k + 1} f \|_{L^2}^2 = \langle ( 1 + a) \nabla \partial_1^{k + 1} p_L, \partial_1^k f \rangle - \langle \partial_2 a \partial_2\partial_1^k p_L,  \partial_1^k f_1 \rangle
  \end{align}
where we have used that $a$  does not depend on $x_1$. Using the definition of $p_L$ we have
   \begin{align}
    \langle ( 1 + a)\nabla \partial_1^{k+ 1 } p_L, \partial_1^kf \rangle 
    &\leq \| 1 + a \|_{L^{\infty}} \| \nabla \partial_1^{k + 1} p_L \|_{L^2} \| \partial_1^k f\|_{L^2} \leq 2\| 1 + a \|_{L^{\infty}} \| \partial_2 a\|_{L^{\infty}} \| \partial_1^{k +1} f \|_{L^2}^2 \notag \\
    \langle \partial_2 a \partial_2 \partial_1^kp_L, \partial_1^k f_1 \rangle 
    &\leq \| \partial_2 a \|_{L^{\infty}} \| \partial_2 \partial_1^k p_L \|_{L^2} \| \partial_k f\|_{L^2} \leq 2\| \partial_2 a \|_{L^{\infty}}^2 \| \partial_1^{k + 1} f \|_{L^2}^2
 \end{align}
where we have used that Poincare's inequality in the $x_1$ variable holds with constant $1$. For the given $\delta$, we can take $\varepsilon_0$ small enough such that if  $\| a \|_{W^{ k + 1,\infty}} = \varepsilon \leq \varepsilon_0$   then
\begin{align}
\label{eq:p1:estimate}
      \frac{1}{2} \frac{d}{dt}\| \partial_1^k f\|_{L^2}^2  + \left(1 - \frac{\delta}{2} \right)  \| \partial_1^{k+ 1} f \|_{L^2}^2\leq 0.
\end{align}
Repeating the same process with $\partial_2^k$ gives
\begin{align*}
  & \frac{1}{2} \frac{d}{dt} \| \partial_2^k f \|_{L^2}^2 + \|(  1 + a) \partial_2^k \partial_1 f\|_{L^2}^2 \notag\\
   &= \sum_{0 < d \leq k} c_{d, k} \langle \partial_2^{d} (1 + a)^2 \partial_2^{k - d} \partial_1^2 f, \partial_2^k f \rangle + \sum_{ 0 \leq d \leq k}c_{d, k} \langle  \partial_2^{d} (1 + a) \partial_2^{k - d}   \nabla \partial_1 p_L, \partial_2^k f \rangle  \notag \\&-\sum_{0 \leq d \leq k} c_{d, k} \langle \partial_2^{d} \partial_2 a \partial_2^{k - d} \partial_2 p_L, \partial_2^k f_1 \rangle = \sum_{0 < d \leq k}T_{1, d} + \sum_{0 \leq d \leq k} (T_{2,d } - T_{3,d})
 \end{align*}
 We now bound $T_{i,d}$:
\begin{align}
\label{eq:T1}
  T_{1,d} &:=  c_{d, k} \langle  \partial_2^{d} (1 + a)^2 \partial_2^{k - d} \partial_1^2 f, \partial_2^k f \rangle \notag\\
   &= -c_{d, k}\langle  \partial_2^{d} (1 + a)^2 \partial_2^{k - d} \partial_1 f, \partial_2^k \partial_1 f \rangle \notag \\
   &\leq c_{d, k} \|  \partial_2^{d} (1 + a)^2  \|_{L^{\infty}} \| \partial_2^{k - d} \partial_1 f \|_{L^2} \| \partial_2^k \partial_1 f\|_{L^2} \notag\\
   &\leq  c_{d, k} c_{j , d} \| \partial_2^j( 1 + a)\|_{L^{\infty}}\|  \partial_2^{d - j}(1 + a)\|_{L^{\infty}} \| \partial_2^{k - d} \partial_1 f\|_{L^2} \| \partial_2^k \partial_1 f\|_{L^2}
\end{align}
where $0 < d \leq k$ and $0 \leq j \leq d$. Since $d$ is never 0, this ensures either $ \partial_2^j( 1 + a) = \partial_2^j a$ or $\partial_2^{d - j}(1 + a) = \partial_2^{d - j}a $ which are smaller than $\varepsilon$ in $L^{\infty}$. Therefore, by choosing  $\varepsilon_0$ sufficiently small, we can take these terms as small as we want. Similarly, 
\begin{align}
\label{eq:T2}
   T_{2, d} &:=  c_{d,k}\langle \partial_2^{d} (1 + a) \partial_2^{k - d}  \nabla \partial_1 p_L , \partial_2^k f \rangle \notag\\
    &\leq   2 c_{d, k} \| \partial_2^{d} (1 + a)\|_{L^{\infty}} \| \partial_2^{k - d} ( \partial_2 a \partial_1 f_2)\|_{L^2} \| \partial_2^k f \|_{L^2} \notag \\
    &\leq  2 c_{d, k} c_{j, k - d} \| \partial_2^{d} ( 1  +a) \|_{L^{\infty}} \| \partial_2^{j + 1 } a \|_{L^{\infty}} \| \partial_2^{k - d - j} \partial_1 f_2 \|_{L^2} \| \partial_2^k f \|_{L^2} \notag\\
    &\leq  2c_{d, k} c_{j, k - d} \| \partial_2^{d} (1 + a) \|_{L^{\infty}} \| \partial_2^{j +1 } a \|_{L^{\infty}} \|\partial_1 f\|_{\dot{H}_0^k}^2 
\end{align}
and
\begin{align}
\label{eq:T3}
   T_{3,d} &:=  c_{d, k }\langle \partial_2^{d + 1} a \partial_2^{k - d + 1} p_L, \partial_2^k f_1 \rangle \notag\\
    &\leq 2c_{d, k} \| \partial_2^{d + 1} a \|_{L^{\infty}} \| \partial_2^{k - d}(\partial_2 a f_2)\|_{L^2} \| \partial_2^k f_1 \|_{L^2} \notag\\
    &\leq 2c_{d, k} c_{j, k - d} \| \partial_2^{d + 1} a\|_{L^{\infty}} \| \partial_2^{j +1 } a\|_{L^{\infty}} \| \partial_2^{k -d -j} f_2\|_{L^2} \| \partial_2^k f \|_{L^2} \notag \\
    &\leq 2c_{d, k} c_{j, k - d} \| \partial_2^{d + 1} a\|_{L^{\infty}} \| \partial_2^{j+ 1 } a\|_{L^{\infty}}  \| \partial_1 f \|_{\dot{H}_0^k}^2.
\end{align}
Combining the estimates for \eqref{eq:T1}, \eqref{eq:T2}, and \eqref{eq:T3} gives
\begin{align}
\label{eq:p2:estimate}
    \frac{1}{2} \frac{d}{dt}\| \partial_2^k f \|_{L^2}^2 + \| (1 + a) \partial_2^k \partial_1 f\|_{L^2}^2 \leq  C \varepsilon \| \partial_1 f\|_{\dot{H}_0^k}^2.
\end{align}
Combining \eqref{eq:p2:estimate} with \eqref{eq:p1:estimate} and taking $\varepsilon_0$ sufficiently small we conclude
\begin{align}
    \frac{1}{2} \frac{d}{dt}
    \|  f \|_{\dot{H}_0^k}^2 \leq  -(1 -\delta)\| \partial_1 f \|_{\dot{H}_0^k} \leq -(1 -\delta) \| f \|_{\dot{H}_0^k}^2
\end{align}
which completes the proof.
\end{proof}

\subsection{Proof of Theorem~\ref{theorem:main}}
The proof is based on the local existence result in Theorem~\ref{thm:local}, and a standard bootstrap argument for the bounds \eqref{eq:thm:stability}. Since $\po b_0 = 0$ we have that $\int_{\TT^2} b_0 dx_1 dx_2 = 0$, and by appealing to \eqref{eq:magnetic:perturbation2a} we see that  $\int_{\TT^2} b (\cdot,t) dx_1 dx_2 = 0$ for all $t\geq 0$. It follows that $\norm{B}_{L^2}^2 = \norm{b}_{L^2}^2 + |\TT|^2 + 2 \int_{\TT^2} b_1 dx_1 dx_2 = \norm{b}_{L^2}^2 + |\TT|^2$. Therefore, \eqref{eq:energy:a} and \eqref{eq:b0:Hm} imply that \begin{align}
\norm{b(\cdot,t)}_{L^2} \leq \norm{b_0}_{L^2} \leq \eps
\label{eq:little:b:L2}
\end{align} 
for all $t\geq 0$. Moreover, \eqref{eq:b0:Hm} and \eqref{eq:Hs:energy:3} show that there exists $T_0>0$ such that for all $t \in [0,T_0]$ we have that $\norm{b(\cdot,t)}_{\dot{H}^m} \leq 2 \eps$. This bound may be combined with \eqref{eq:little:b:L2} to conclude that the bounds \eqref{eq:thm:stability} hold on $[0,T_0]$, with all inequalities being {\em strict inequalities}. Due to the local existence result in Theorem~\ref{thm:local} via a standard continuity argument we may thus define a maximal time $T_* \in [T_0,\infty]$ such that the estimates \eqref{eq:thm:stability} hold on $[0,T_*)$. Our goal is to show that $T_* = \infty$. In order to achieve this we show that if \eqref{eq:thm:stability} hold on $[0,T]$ for some $T>0$, then we may a posteriori show that these bounds in fact hold with constants $3\eps$ instead of $4\eps$ in \eqref{eq:f:Hk}--\eqref{eq:b:Hm}; this then shows $T_* = \infty$. The rest of the proof is dedicated to establishing this implication, and so we fix a time interval $[0,T]$, and we assume throughout that \eqref{eq:thm:stability} hold. We recall and use the notation in \eqref{eq:a:f:w:def}.

\subsubsection{Estimates for the nonlinear terms $N$ and $N'$}
Under the standing assumptions, we estimate the nonlinear terms $N(f,w)$ defined in \eqref{def:nonlinearity} and $N'(f,w)$ defined in \eqref{eq:b:zero:mode} and claim that 

\begin{equation}
\label{eq:nonlinear:bound:bbar}
    \| N(f,w)(\cdot,t) \|_{H^k} \leq   \varepsilon^{2}   e^{-\frac{3}{2}(1 - \delta) t}
\end{equation}
and 
\begin{equation}
\label{eq:nonlinear:bound:btilde}
    \|  N'(f,w)(\cdot,t) \|_{H^{k + 2}} \leq \varepsilon^{2} e^{-(1 - \delta) t} . 
\end{equation}

Prior to establishing \eqref{eq:nonlinear:bound:bbar} and \eqref{eq:nonlinear:bound:btilde}, we claim that the pressure terms in \eqref{def:lin:pressure} and \eqref{def:nonlin:pressure} satisfy the bounds  
\begin{subequations}
\begin{align}
\label{eq:lin:pressure:bound}
    \| p_L \|_{H^{k + \beta}} 
    & \lesssim \|a  \|_{H^{k + \beta}} \|f \|_{H^{k}}^{\frac{m - k - \beta + 1 }{m - k}} \|f \|_{H^{m}}^{\frac{\beta - 1}{m - k}} \\
\label{eq:nonlin:pressure:bound}
     \| p_N \|_{H^{k + \beta}} 
     &\lesssim \| f \|_{H^{k }}^{\frac{2( m -k - \beta + 1)}{m-k}} \| f \|_{H^{m}}^{\frac{2(\beta - 1)}{m- k}} 
\end{align}
\end{subequations}
for $1 \leq \beta\leq 10$. The estimate for the linear pressure follows directly from \eqref{def:lin:pressure},  the fact that $H^k$ is an algebra, and interpolation:
\begin{align*}
    \| p_L \|_{H^{k + \beta}} \lesssim \| \partial_2 a  f_2 \|_{H^{k + \beta - 1}} \lesssim \|a \|_{H^{k + \beta}} \| f_2 \|_{H^{k + \beta- 1}} 
    \lesssim \|a  \|_{H^{k + \beta}} \| f\|_{H^{k}}^{\frac{m - k - \beta + 1 }{m - k}} \|f \|_{H^{m}}^{\frac{\beta - 1}{m - k}} \, .
\end{align*}
Similarly, from \eqref{def:nonlin:pressure} we have
\begin{align*}
    \| p_N \|_{H^{k + \beta}} 
    &\lesssim \| (\partial_1 f_1)^2 + \partial_1 f_2 \partial_2 f_1  \|_{H^{k + \beta -2 }}
    \lesssim \|f \|_{H^{k + \beta-1}}^2     
    \lesssim \| f \|_{H^{k }}^{\frac{2( m -k - \beta + 1)}{m-k}} \| f \|_{H^{m}}^{\frac{2(\beta - 1)}{m- k}}.
\end{align*}

Next, we claim that the velocity field $w$ from \eqref{def:u:bar} satisfies the estimate 
\begin{align}
\label{eq:velocity:bound}
    \|w \|_{H^{k  + \beta}} 
    & \lesssim \| a \|_{H^{k + \beta + 1}} \| f \|_{H^k}^{\frac{ m-k - \beta -1}{ m-k}}  \| f \|_{H^m}^{\frac{\beta + 1}{m - k}}   + \| f \|_{H^k}^{\frac{2(m - k - \beta -1)}{m - k}}  \| f \|_{H^m}^{\frac{2(\beta + 1)}{m - k}} 
\end{align}
for $\beta \geq 0$. This bound follows from \eqref{def:u:bar}, the previously established bounds \eqref{eq:nonlin:pressure:bound}--\eqref{eq:nonlin:pressure:bound}, and an algebra + interpolation argument:
\begin{align*}
    \| w\|_{H^{k + \beta}} 
    &\lesssim \|a \partial_1 f \|_{H^{k + \beta}} + \|\partial_2a f_2 \|_{H^{k + \beta}} + \| f \cdot \nabla f \|_{H^{k + \beta}} + \| \nabla p_L \|_{H^{k + \beta}} + \| \nabla p_N \|_{H^{k + \beta}} \notag\\
   &\lesssim \| a \|_{H^{k + \beta + 1 }} \| f \|_{H^{k  +\beta + 1}} + \| f \|_{H^{k + \beta + 1}}^2 + \| p_L\|_{H^{k + \beta + 1}} + \| p_L\|_{H^{k + \beta + 1}} \notag \\
   &\lesssim \| a \|_{H^{k + \beta + 1}} \| f \|_{H^k}^{\frac{ m - k - \beta - 1}{ m - k}} \| f \|_{H^m}^{\frac{ \beta + 1 }{m - k}} 
   +  \| f \|_{H^k}^{\frac{2( m - k - \beta - 1)}{ m - k}} \| f \|_{H^m}^{\frac{2 (\beta + 1) }{m - k}} \notag \\
   &\qquad + \| a \|_{H^{k + \beta + 1}} \| f \|_{H^k}^{\frac{ m - k - \beta }{m - k}}\| f \|_{H^m}^{\frac{\beta}{m- k}} + \| f \|_{H^k}^{\frac{ 2(m - k - \beta) }{m - k}}\| f \|_{H^m}^{\frac{2\beta}{m- k}}\,.
   \end{align*}
The bound \eqref{eq:velocity:bound} follows by  using the Poincar\'e inequality (recall that $f$ has zero mean on $\TT^2$).

With \eqref{eq:nonlin:pressure:bound}--\eqref{eq:nonlin:pressure:bound} and \eqref{eq:velocity:bound} available, we next give the proof of \eqref{eq:nonlinear:bound:bbar}. The right side of \eqref{def:nonlinearity} contains ten terms, and as such we estimate
 \begin{align*}
    \| N(f,w)\|_{H^k} 
    \leq  N_1 + \ldots + N_{10} \,,
  \end{align*}
where   
 \begin{align*}
      N_1&:=   \| a \partial_1 \pn (f \cdot \nabla f ) \|_{H^k} \lesssim \| a \|_{H^k}  \| f \|_{H^{k + 2}}^2 \lesssim \| a \|_{H^k} \| f \|_{H^k}^{2- \frac{4}{m-k}}\| f \|_{H^m}^{\frac{4}{m - k}} \\
     N_2 &:=  \| a \partial_1 \pn \nabla p_N \|_{H^k} \lesssim \| a \|_{H^k} \| p_N \|_{H^{k + 2}} \lesssim  \| a \|_{H^k} \| f \|_{H^{k }}^{2 - \frac{2}{m-k}} \| f \|_{H^{m}}^{\frac{2}{m- k}}\\
  N_3&:=  \| \pn(f \cdot \nabla w )\|_{H^k} \lesssim \| f \|_{H^k} \| w\|_{H^{k  +1 }} 
  \les \| a \|_{H^{k + 2}} \| f \|_{H^k}^{2 - \frac{2}{ m-k}}  \| f \|_{H^m}^{\frac{2}{m - k}}   + \| f \|_{H^k}^{3 - \frac{4}{m - k}}  \| f \|_{H^m}^{\frac{4}{m - k}} \\
  N_4&:= \|\pn(w \cdot \nabla f) \|_{H^k}  
      \lesssim \| w\|_{H^k} \| f \|_{H^{k + 1}} 
      \les \| a \|_{H^{k +  1}} \| f \|_{H^k}^{2 - \frac{2}{ m-k}}  \| f \|_{H^m}^{\frac{2}{m - k}}   + \| f \|_{H^k}^{3 - \frac{3}{m - k}}  \| f \|_{H^m}^{\frac{3}{m - k}}   \notag \\
   N_5&:= \|\partial_2 \po(f_1 f_2)\partial_1 f \|_{H^k} \lesssim \|f\|_{H^{k +1}}^3 \lesssim \| f \|_{H^k}^{3 -  \frac{3}{m-k} }\| f\|_{H^m}^{\frac{3}{m-k}}\\
   N_6 &:= 2\|\pn(f \cdot \nabla \partial_1 f)\|_{H^k} \lesssim \| f \|_{H^{k + 2}}^2 \lesssim \| f \|_{H^k}^{2 - \frac{4}{ m - k}} \| f \|_{H^m}^{\frac{4}{m - k}}\\
    N_7 &:= \|\nabla \partial_1 \pn p_N\|_{H^k} \lesssim \| p_N \|_{H^{k + 2}} \lesssim \| f \|_{H^{k }}^{2 - \frac{2}{m-k}} \| f \|_{H^{m}}^{\frac{2}{m- k}}\\
    N_8 &:= \|\partial_2^2 \po(f_1 f_2)f_2\|_{H^k} \lesssim \| f \|_{H^{k + 2}}^3 \lesssim \| f \|_{H^k}^{3 - \frac{6}{m -k}} \| f \|_{H^m}^{\frac{6}{m-k}}\\
    N_9 &:= \|\partial_2 a \pn(f \cdot \nabla f_2)\|_{H^k} \lesssim \| a \|_{H^{k + 1}} \| f \|_{H^{k + 1}}^2 \lesssim \| a \|_{H^{k + 1}} \| f \|_{H^k}^{2 - \frac{2}{ m -k }} \| f \|_{H^m}^{\frac{2}{m -k }}\\
    N_{10} &:= \|\partial_2 a \partial_2 \pn p_N\|_{H^k} \lesssim \| a \|_{H^{k + 1}} \| p_N\|_{H^{k + 1}} \lesssim \| a \|_{H^{k + 1}}\| f \|_{H^{k }}^2\,,
\end{align*}
where the implicit constants only depend on $m$ and $k$. 
At this point we use the assumption that $ m- k \geq 9$, the standing assumption \eqref{eq:thm:stability:alt}, and 
% \footnote{$ N_6$ has the worst decay and is largest.}, 
take $\varepsilon$ sufficiently small depending on $\delta$ to obtain that \eqref{eq:nonlinear:bound:bbar} holds.

In a similar fashion, we estimate the nonlinear term in \eqref{eq:b:zero:mode} as
\begin{align*}
    \| N'(f,w) \|_{H^{k + 2}} \leq N_1' + N_2' + N_3'
\end{align*}
with 
\begin{align*}
 N_1' &:= 2\| \partial_2\po  (f_2 \partial_1 f_1) \|_{H^{k+2}} 
 \lesssim \| f \|_{H^{k + 4}}^2  
 \lesssim  \| f \|_{H^{k}}^{2-\frac{8}{ m - k}} \| f \|_{H^m}^{\frac{8}{m - k}}\\
 N_2' &:= \| \partial_2\po ( f_2 w_1) \|_{H^{k + 2}} \lesssim \| f \|_{H^{k + 3}}\| w \|_{H^{k + 3}} 
 \lesssim \| a \|_{H^{k + 4}} \| f \|_{H^k}^{2 - \frac{7}{ m-k}}  \| f \|_{H^m}^{\frac{7}{m - k}}   + \| f \|_{H^k}^{3 -\frac{11}{m - k}}  \| f \|_{H^m}^{\frac{11}{m - k}}   \\
     N_3'  &:= \| \partial_2\po ( w_2 f_1) \|_{H^{k + 2}}  
     \les \norm{w}_{H^{k+3}} \norm{f}_{H^{k+3}}
      \lesssim \| a \|_{H^{k + 4}} \| f \|_{H^k}^{2 - \frac{7}{ m-k}}  \| f \|_{H^m}^{\frac{7}{m - k}}   + \| f \|_{H^k}^{3 -\frac{11}{m - k}}  \| f \|_{H^m}^{\frac{11}{m - k}}  \,,
\end{align*}
where the implicit constant depends only on $m$ and $k$. Once again, using  $ m- k \geq 9$,  the standing assumption \eqref{eq:thm:stability:alt}, and  the bound $\|a(t) \|_{H^{k  +4}} \lesssim \|a(t) \|_{H^{m}} \lesssim e^{ \varepsilon t}$, after taking $\varepsilon$ sufficiently small depending on $\delta$ we have that  \eqref{eq:nonlinear:bound:btilde} holds.

\subsubsection{Closing the \eqref{eq:f:Hk} bootstrap}
This argument is based on the Duhamel formula, which allows us via \eqref{eq:full:equation} to write
\begin{equation}
    \label{eq:Duhamel}
    f(t) = e^{Lt}f_0 + \int_0^t e^{L(t -s)} N(f,w)(s) ds \,.
\end{equation}
Next, we note that $a$ is a function defined on $\TT$, and thus by the Sobolev embedding we have $H^{k +2} \subset W^{k + 1, \infty}$, which shows that \eqref{eq:thm:stability:alt} implies \eqref{eq:Lagavulin}; thus,  we may apply Proposition~\ref{prop:linear:decay}. Applying the  $\dot{H}^k$  norm to \eqref{eq:Duhamel}, and appealing to \eqref{eq:nonlinear:bound:bbar}, we arrive at
\begin{align}
    \| f(t)\|_{\dot{H}^k} 
    &\leq \varepsilon e^{-(1 - \delta)t} + \int_0^t e^{ -(1 - \delta) (t -s)}\| N(f,w)(s) \|_{\dot{H}^k} ds\notag\\
    &\leq \varepsilon e^{-(1 - \delta)t} + \eps^{2} \int_0^t e^{ -(1 - \delta) (t -s)} e^{-\frac 32 (1-\delta) s} ds \notag\\
    &\leq \varepsilon e^{-(1 - \delta)t} \left( 1 + \eps  \int_0^t  e^{-\frac 12 (1-\delta) s} ds \right)\notag\\
&\leq    2 \varepsilon e^{- (1 - \delta) t}\,,
\end{align}
once $\eps$ is chosen sufficiently small with respect to $\delta$. This bound improves on \eqref{eq:f:Hk}, as desired.

\subsubsection{Closing the \eqref{eq:a:H:k+2} bootstrap}
By assumption, we have that $a|_{t=0}=0$. Integrating the evolution equation for $a$ in \eqref{eq:b:zero:mode}, and appealing to \eqref{eq:nonlinear:bound:btilde}, we thus deduce that 
\begin{align*}
\norm{a(\cdot,t)}_{H^{k+2}} \leq \int_0^t \norm{ N'(f,w)(\cdot,s)}_{H^{k+2}}  ds \leq \eps^{2} \int_0^t e^{-(1-\delta)s} ds \leq \eps
\end{align*}
for all $t\in [0,T]$, upon choosing $\eps$ to be sufficiently small in terms of $\delta$. This improves \eqref{eq:a:H:k+2} by a constant factor, as desired.

\subsubsection{Closing the \eqref{eq:b:Hm} bootstrap}
From \eqref{eq:blowup:criterion}, \eqref{eq:velocity}, \eqref{eq:b0:Hm}, and the fact that $\norm{B}_{\dot{H}^s} = \norm{b}_{\dot{H}^s}$, we deduce that for a constant $C_m$ which only depends on $m$, we have
\begin{align}
\norm{b(\cdot,t)}_{\dot{H}^m}^2 
&\leq \eps^2  \exp\left( C_{m} \int_0^t \norm{\nabla v(\cdot,s)}_{L^\infty} + \norm{\nabla \p_1 b(\cdot,s)}_{L^\infty} + \norm{\nabla b(\cdot,s)}_{L^\infty}^2 ds \right) \,.
\label{eq:b:Hm:temp:1}
\end{align}
From \eqref{eq:thm:stability:alt}, \eqref{eq:velocity:decomposition}, \eqref{eq:velocity:bound}, and the fact that $k > d/2+1= 2$, we note that 
\begin{align}
\int_0^t \norm{\nabla v(\cdot,s)}_{L^\infty} ds 
&\les \int_0^t \norm{w(\cdot,s)}_{H^k} + \norm{f(\cdot,s)}_{H^{k+1}}^2 ds \notag\\
&\les \int_0^t \norm{a(\cdot,s)}_{H^{k+1}} \norm{f}_{H^k}^{1 - \frac{1}{m-k}} \norm{f}_{H^m}^{\frac{1}{m-k}} + \norm{f}_{H^k}^{2 - \frac{2}{m-k}} \norm{f}_{H^m}^{\frac{2}{m-k}}  ds \notag\\
&\les \eps^{2} \int_0^t e^{- \frac 89 (1-\delta) s + \frac 29 \eps s } ds \les \eps 
\,,
\label{eq:b:Hm:temp:2}
\end{align}
if $\eps$ is sufficiently small with respect to $\delta$. Returning to \eqref{eq:b:Hm:temp:1}, we see that $\nabla \p_1 b = \nabla \p_1 f$, and since $k > 2 + d/2 = 3$ and $f$ has zero mean on $\TT^2$, we have that \eqref{eq:thm:stability:alt} implies
\begin{align}
\int_0^t   \norm{\nabla \p_1 b(\cdot,s)}_{L^\infty}  ds
\les \int_0^t   \norm{f(\cdot,s)}_{H^k}  ds
\les \eps \int_0^t  e^{-(1-\delta) s} ds 
\les \eps^{\frac 12}\,.
\label{eq:b:Hm:temp:3}
\end{align}
Lastly, for the third term in \eqref{eq:b:Hm:temp:1} we similarly note that 
\begin{align}
\int_0^t   \norm{\nabla  b(\cdot,s)}_{L^\infty}^2  ds
\les \int_0^t  \norm{a(\cdot,s)}_{H^k}^2 +  \norm{f(\cdot,s)}_{H^k}^2  ds
\les \eps^2 t + \eps \,,
\label{eq:b:Hm:temp:4}
\end{align}
where the implicit constant only depends on $m$ and $k$. By combining \eqref{eq:b:Hm:temp:1}--\eqref{eq:b:Hm:temp:4} we thus obtain that there exists a constant $C_{m,k} >0$, which only depends on $m$ and $k$, such that 
\begin{align*}
\norm{b(\cdot,t)}_{\dot{H}^m}^2 
\leq \eps^2  \exp\left( C_{m,k} (\eps^2 t + \eps^{\frac 12}) \right)
\leq \eps \exp( \eps t)
\end{align*}
upon taking $\eps$ to be sufficiently small, solely in terms of $m$ and $k$. This bound improves on \eqref{eq:b:Hm}, as desired.

\subsubsection{Relaxation in the infinite time limit}
We recall that the total velocity field has zero mean on $\TT^2$ and is given from \eqref{eq:velocity}--\eqref{def:perturbation:velocity} as $u = \p_1 b + v$, and $v$ is computed from $f$ and $w$ via \eqref{eq:velocity:decomposition}. Since $\p_1 b = \p_1 f$, the fact that $\norm{u(\cdot,t)}_{H^k} \to 0$ as $t\to \infty$, exponentially fast, now follows from \eqref{eq:thm:stability:alt} and \eqref{eq:velocity:bound}:
\begin{align*}
 \norm{u(\cdot,t)}_{H^{k-1}} 
 &=  \norm{u(\cdot,t)}_{\dot{H}^{k-1}} 
 \les \norm{f}_{\dot{H}^k} + \norm{w}_{\dot{H}^{k}} + \norm{f_1 f_2}_{\dot{H}^k} \notag\\
 &\les \eps e^{-(1-\delta)t} + \eps^3 e^{-  \frac{(1-\delta)(m-k-1) -\eps}{m-k} t } + \eps^2 e^{- \frac{2(1-\delta) (m-k-1) - 2\eps}{m-k} t } + \eps^2 e^{-2(1-\delta)t} \notag\\
 &\les \eps e^{- \frac{1-\delta}{2}t} \,.
\end{align*}
In order to conclude the proof of the theorem, we note that in view of \eqref{eq:MHD:a} and the fact that $k\geq 4$, we have 
\begin{align*}
 \norm{\p_t b}_{H^{k-2}} = \norm{\p_t B}_{H^{k-2}}
 \les \norm{B \otimes u}_{H^{k-1}}
 \les \norm{B}_{H^{k-1}} \norm{u}_{H^{k-1}}
 \les (1 + \norm{b}_{H^k}) \norm{u}_{\dot{H}^{k-1}}
\end{align*}
and thus in view of \eqref{eq:f:Hk}--\eqref{eq:a:H:k+2} we obtain
 \begin{align*}
 \norm{\p_t b}_{H^{k-2}}
 \les   \eps e^{- \frac{1-\delta}{2}t} \,.
 \end{align*}
The strong convergence $\lim_{t \to \infty} b(\cdot,t) = \bar b$, with respect to the $H^{k-2}$ norm, for an incompressible vector field $\bar b$ which has norm $\leq 4 \eps$, now follows from the above estimate and the fundamental theorem of calculus in time. The corresponding limiting relaxation state for the total magnetic field is then $\bar B = e_1+ \bar b$.

\section{Nonlinear instabilities in 3D}
\label{sec:3D:instability}

In this section we consider a class of {\em two-and-a-half dimensional} exact solutions of the three-dimensional MRE system \eqref{eq:MRE}, when $\gamma = 0$, and show that for suitable choices of initial conditions, these solutions exhibit infinite time growth of gradients. These examples draw on an analogy with the 3D Euler equation, for which Yudovich~\cite{Yudovich74,Yudovich00} has constructed similar solutions. Theorem~\ref{thm:growth} below gives an example in which magnetic relaxation holds with respect to the $L^2$ norm, but fails with respect to the $H^1$ norm.  Furthermore, as in the work of Elgindi and Masmoudi~\cite{ElgindiMasmoudi20} we construct examples where the magnetic current grows exponentially in time, for all time.

To fix notation, for any vector $x \in \RR^3$, we denote by $x_H$ its first two {\em horizontal} components, i.e. $x_H = (x_1,x_2)$. We also write $\div_H = \nabla_H \cdot$ where $\nabla_H = (\partial_1, \partial_2)$. 

\subsection{Euler examples}
We recall from~\cite{Yudovich74,Yudovich00} the following two-and-a-half dimensional solution of 3D Euler, which exhibits infinite time growth of the vorticity. 

The setting is as follows. Consider any stationary state $v = v(x_H)$ of the  2D  Euler equations on $\TT^2$. These stationary states may be written as $ v =  \nabla_H^\perp \phi$, where the periodic stream function $\phi \colon \TT^2 \to \RR$ satisfies $  \Delta_H \phi = F(\phi)$ for a sufficiently smooth $F$. Then, an exact solution of the 3D Euler system is given by 
\begin{align}
u(x,t) = (v(x_H), g(x_H,t))
\label{eq:3D:Euler:exact}
\end{align}
where the function $g \colon \TT^2 \times \RR_+ \to \RR$ satisfies the transport equation
\begin{align}
\p_t g + (v \cdot \nabla_H) g = 0 \,. 
\label{eq:g:Euler}
\end{align}
Indeed, one may verify that $\p_t u + u \cdot \nabla u = ( v \cdot \nabla_H v, \p_t g + v\cdot\nabla_H g) = (- \nabla_H p, 0) = - \nabla p$, where $p=p(x_H)$ is the pressure associated to the steady solution $v$.

{\bf Shear flow}. \,
When $v$ is a shear flow, such as 
$$
v(x_H) = \left(V(x_2),0 \right) 
$$
for a smooth function $V \colon \TT \to\RR$, the solution of \eqref{eq:g:Euler} is explicit in terms of its initial datum $g_0$:
\begin{align}
g(x_H,t) = g_0(x_1 - t V(x_2),x_2) 
\,.
\label{eq:g:Euler} 
\end{align}
Thus, combining \eqref{eq:3D:Euler:exact} with \eqref{eq:g:Euler}, we are lead to the following exact solution of 3D Euler:
\begin{align}
u(x,t) = \left( V(x_2), 0 , g_0(x_1 - t V(x_2),x_2) \right) \,.
\label{eq:Yudovich}
\end{align}
Even though $u(\cdot,t)$ remains bounded with respect to the $L^\infty$ norm, the vorticity $\omega = \nabla \times u$ has the property that its first component is given by
$$
\omega_1(x,t) = (\p_2 u_3 - \p_3 u_2)(x,t) = - t V'(x_2) (\p_1 g_0)( x_1 - t V(x_2),x_2) + (\p_2 g_0)( x_1 - t V(x_2),x_2) \,.
$$
It is clear that for suitable choice of the initial datum $g_0$, and for $V \not \equiv $ constant, we have that $\norm{\omega_1(\cdot,t)}_{L^\infty} \gtrsim t$ as $t\to\infty$. 
As such, in~\cite{Yudovich00}, Yudovich calls the solution given in \eqref{eq:Yudovich} as weakly nonlinearly unstable.

{\bf Hyperbolic flow}. \, 
In analogy with the above example, Elginidi and Masmoudi~\cite{ElgindiMasmoudi20} consider the stationary solution $v = v(x_H)$ appearing in \eqref{eq:3D:Euler:exact} to be an eigenfunction of the Laplacian which displays hyperbolic dynamics near the separatrix; more precisely, they consider
\begin{equation}
    v = \nabla_H^\perp \left( \sin x_1 \sin x_2 \right)
    = \begin{pmatrix}
          -\sin x_1 \cos x_2 \\
          \cos x_1 \sin x_2
           \end{pmatrix}.
\end{equation}
In this case,  the solution $g$ of the transport equation \eqref{eq:g:Euler} is again bounded, but its derivative in the $x_1$ direction, restricted to the separatrix $\{ x_2 = 0\}$ satisfies the equation
\begin{align*}
\p_t (\p_1 g)|_{x_2=0} - \sin x_1 \p_1 (\p_1 g)|_{x_2=0} = \cos x_1 (\p_1 g)|_{x_2=0}
\,,
\end{align*}
and as such we have that $(\p_1 g)(0,0,t) = (\p_1 g)(0,0,0) e^{t}$. Thus, in this situation the solution $u$ of 3D Euler given by \eqref{eq:3D:Euler:exact} exhibits exponential growth with respect to time of the first component of the vorticity
\begin{align*}
\norm{\omega_1(\cdot,t)}_{L^\infty} \geq e^t  |\omega_1(0,0,0)|\,.
\end{align*}

\subsection{MRE examples}
The example of~Yudovich outlined above, has a direct correspondent for the 3D MRE system. The main observation is that 
an exact solution of the 3D MRE equations \eqref{eq:MRE} with $\gamma = 0$ is given by
\begin{align}
B = \left( v , g \right), \qquad  u = \left(0, 0,  (v \cdot \nabla_H) g \right)
\label{eq:3D:ansatz}
\end{align}
where $g = g(x_H,t) \colon \TT^2 \times \RR_+ \to \RR$ satisfies the rank $1$ diffusion equation
\begin{align}
\partial_t g = ( v \cdot \nabla_H)^2 g \label{eq:g:evo}
\,.
\end{align}
In order to verify this, we start from the ansatz \eqref{eq:3D:ansatz} and the fact that $g$ is independent of $x_3$, to immediately see that $B$ and $u$ are divergence free, so that \eqref{eq:div:free} holds. When $\gamma = 0$, and with the ansatz~\eqref{eq:3D:ansatz}, the first two components of \eqref{eq:u:def} become $ 0 = v\cdot\nabla_H v + \nabla_H p$; this identity holds because $v$ was chosen to be an exact stationary state of the 2D Euler equations. On the other hand, the third component of \eqref{eq:u:def} reads $u_3 = v \cdot \nabla_H g$, which justifies the definition of $u_3$ in \eqref{eq:3D:ansatz}. Lastly, in view of \eqref{eq:3D:ansatz} we have $u\cdot \nabla B = 0$, since $v, g$ are independent of $x_3$, and $B \cdot \nabla u = e_3 (v \cdot \nabla_H)(v\cdot \nabla_H) g$. Therefore, \eqref{eq:g:evo} ensures that \eqref{eq:B:evo} holds, as claimed.\

Comparing the MRE evolution of $g$ in~\eqref{eq:g:evo} to the Euler evolution of $g$ in~\eqref{eq:g:Euler}, we see that the main difference is that $g$ does not solve a transport equation, but rather a {\em rank $1$ diffusion equation}. Nonetheless, the gradient of $g$ may still exhibit infinite time growth, which is what we show next.

{\bf Shear flow}. \,  The evolution equation \eqref{eq:g:evo} is particularly easy to solve if the 2D Euler steady state $v$ is chosen to be a shear flow. As such, consider
\begin{align}
v(x_H) = \left(V(x_2),0 \right) 
\label{eq:v:shear}
\end{align}
for a smooth $\TT$-periodic scalar function $V$. With \eqref{eq:g:evo}, the evolution \eqref{eq:g:evo} becomes
\begin{align}
\p_t g = V^2(x_2) \partial_{11} g
\label{eq:b1:evo} 
\end{align}
which is a heat equation in the $(x_1,t)$ variables, with viscosity coefficients that depend on $x_2$. In particular, if we choose an initial datum $g_0$ which is just a function of $x_1$ and such that its mean-free part is an eigenfunction of $\p_{11}$, i.e. 
\begin{align}
- \p_{11} g_0(x_1) = \lambda^2 \left( g_0(x_1) - \fint_{\TT} g_0 \right)
\label{eq:g0:explicit}
\end{align}
for some $\lambda>0$,  we have that the solution of  equation \eqref{eq:b1:evo} is given by 
\begin{align}
g(x_1,x_2,t) =\fint_{\TT} g_0  + \exp\left(- \lambda^2 V^2(x_2) t \right) \left( g_0(x_1)   - \fint_{\TT} g_0  \right)
\,.
\label{eq:g:explicit}
\end{align}
We have thus shown that if the 2D Euler steady state is given by the shear flow in \eqref{eq:v:shear}, and if the initial datum for the third component of the magnetic field is a function that satisfies \eqref{eq:g0:explicit}, then for $g$ given by \eqref{eq:g:explicit} the functions 
\begin{align}
 B(x,t) = \left (V(x_2), 0, g(x_1,x_2,t)\right), \qquad u(x,t) = \left(0, 0, V(x_2) \p_1 g(x_1,x_2,t) \right)
 \label{eq:more:general:sol}
\end{align}
are exact solutions of the 3D MRE equations \eqref{eq:MRE} with $\gamma = 0$. 

In particular, the above example shows that solutions of \eqref{eq:MRE} with $\gamma = 0$ and $d=3$ exhibit infinite time growth of gradients, even if the initial datum is a small perturbation of the steady state $B = e_3$ and $u=0$:  
\begin{theorem}[\bf Example of solution with infinite time growth]
\label{thm:growth}
 There exists an incompressible initial condition $(B_0 ,u_0)$  such that $B_0 - e_3 = \OO(\eps)$ and $u_0 = \OO(\eps^2)$ in arbitrary strong topologies (e.g.~real-analytic), and such that the unique solution $B$ of the 3D MRE equation \eqref{eq:MRE} with this initial datum and $\gamma = 0$ satisfies $\norm{\nabla B(\cdot,t)}_{L^2} = \OO(\eps^{3/2} t^{1/4})$ as $t\to \infty$.
\end{theorem}
\begin{proof}[Proof of Theorem~\ref{thm:growth}]
In \eqref{eq:more:general:sol}, take $V(x_2) = \eps \sin(x_2)$, and $g_0(x_1) = 1 + \eps \cos(x_1)$. This corresponds to the initial conditions
\begin{align*}
B_0 = e_3 + \eps (\sin(x_2), 0 , \cos(x_1))
\qquad \mbox{and} \qquad
u_0 = - \eps^2 (0,0,\sin(x_2) \sin(x_1))
\end{align*}
which are clearly divergence free, and smooth. Due to the local existence and uniqueness theorem, we know that the solution is given by \eqref{eq:more:general:sol}, where by \eqref{eq:g:explicit} we have
\begin{align*}
g(x_1,x_2,t) =  1 + \eps  \cos(x_1) \exp\left(- \eps^2 \sin^2(x_2) t \right)  \,.
\end{align*}
In particular, by \eqref{eq:more:general:sol} and the above formula, we have that 
\begin{align*}
\p_{2} B_3(x,t) =  - 2 t \eps^3 \sin(x_2) \cos(x_2)  \cos(x_1) \exp\left(- \eps^2 \sin^2(x_2) t \right)
\end{align*}
and so we may explicitly compute  
\begin{subequations}
\label{eq:explicit:growth}
\begin{align}
\lim_{t\to\infty} \frac{1}{C_1 \eps^{3/2}  t^{1/4}}  \norm{\partial_2 B_3(\cdot,t)}_{L^2_{x_1,x_2}} &= 1 \\
\lim_{t\to\infty} \frac{1}{C_2 \eps^2 t^{1/2} }  \norm{\partial_2 B_3 (\cdot,t)}_{L_{x_1}^2L^\infty_{x_2} } &= 1 
\end{align}
\end{subequations}
where $C_1 = (2\pi^3)^{1/4}$, $C_2 = (2 \pi^2/e)^{1/2}$. 
Therefore, we have   $\norm{\nabla B(\cdot,t)}_{L^2} = \OO(\eps^{3/2} t^{1/4})$ as $t\to \infty$.
\end{proof}

\begin{remark}[\bf Asymptotic behavior with respect to weak topologies]
\label{rem:asymptotic:L2}
While the $\dot{H}^1$ norm of the solution $B$ defined by \eqref{eq:more:general:sol} is growing without bound as time goes to infinity, we emphasize that its $L^2$ norm remains uniformly bounded.
In fact, for the solution in \eqref{eq:more:general:sol} we have the following asymptotic {\em pointwise} description 
\begin{align}
\lim_{t\to \infty} B(x,t) 
\to \bar B(x) :=
\begin{cases}
\left(V(x_2), 0 ,  \fint_{\TT}g_0 \right)\,, & \mbox{if } V(x_2) \neq 0 \\
\left(V(x_2), 0 ,   g_0(x_1) \right)\,, & \mbox{if } V(x_2) \neq 0 
\end{cases}
\label{eq:bar:B:example}
\end{align}
and $\lim_{t\to \infty} u(x,t) =0$.
This is thus an example of magnetic relaxation: $B$ converges to a magnetostatic equilibrium $\bar B$,  while $u$ converges to $0$, as $t\to \infty$. However, this relaxation holds with respect to {\em weak topologies} only (e.g.~$L^2$), and weak nonlinear instability takes place in stronger topologies (e.g.~$H^1$).
\end{remark}

\begin{remark}[\bf The emergence of current sheets in the infinite time limit]
\label{rem:current}
We note that even though the initial datum in the example of Remark~\ref{rem:asymptotic:L2} is smooth, namely $B_0 = (V(x_2),0,g_0(x_1))$, the (weak) limiting magnetostatic equilibrium $\bar B$ may contain discontinuities in the vertical direction. For instance, take $V(x_2) = {\bf 1}_{x_2\in[-\pi/2,\pi/2]} \cos^2(x_2)$, and $g_0(x_1) = \sin(x_1)$. Then we have that the $\bar B$ vector field defined in \eqref{eq:bar:B:example} is given by 
\begin{align*}
\bar B(x) =
\begin{cases}
\left(\cos^2(x_2), 0 ,  0 \right)\,, & \mbox{if } |x_2| \leq \pi/2 \neq 0 \\
\left(0, 0 ,   \sin(x_1) \right)\,, & \mbox{if } |x_2|> \pi/2 \neq 0 
\end{cases}
\,,
\end{align*}
which clearly contains a discontinuity along the planes $\{x\in \TT^3 \colon x_2 = \pm \pi/2\}$. The associated current field $\bar j = \nabla \times \bar B$ is given by the sum of a bounded piece ${\bf 1}_{|x_2|<\pi/2} (0,0,\sin(2x_2)) + {\bf 1}_{|x_2|>\pi/2} (0,-\cos(x_1),0)$, and a singular part which a Dirac mass supported on the planes $x_2 = \pm \pi/2$ and has amplitude $(\pm \sin(x_1), 0, 0)$.
\end{remark}

{\bf Hyperbolic flow}. \,
While Theorem~\ref{thm:growth} exhibits solutions whose magnetic current grows algebraically in time as $t\to \infty$, following \cite{ElgindiMasmoudi20} we may show that if $v$ is chosen to be the cellular flow
\begin{equation}
\label{eq:hyperbolic:flow}
  v = \nabla_H^\perp \left(\sin x_1 \sin x_2 \right)
     = \begin{pmatrix}
          -\sin x_1 \cos x_2 \\
          \cos x_1 \sin x_2
           \end{pmatrix}
           \,,
\end{equation}
then we can in fact find solutions $g$ of \eqref{eq:g:evo}, and hence of the MRE equations, which exhibit exponential growth of their gradients. This is the worst growth that they can sustain, given that \eqref{eq:g:evo} is an equation linear in $g$.

\begin{theorem}[\bf Example of solution with exponential growth]
\label{thm:exponential:growth}
Let  $v$ be as in \eqref{eq:hyperbolic:flow}.
For any initial data $g_0 \in H^{k}$, with $k\geq 3$ an integer, which satisfyies $\nabla_H g_0 (0,0) \neq 0$, the unique solution of the 3D MRE equation \eqref{eq:MRE} with initial data $B_0 = (v , g_0)$ and $u_0 = (0, 0, v\cdot\nabla_H g_0)$, satisfies 
\begin{equation}
\label{eq:B:exp:growth}
\abs{\nabla_H g_0 (0,0)} e^t \leq \| \nabla B(\cdot,t)\|_{L^{\infty}} \leq C \norm{B_0}_{H^k} e^{C t}
\,,
\end{equation}
where $C>0$ is a constant which only depends on $k$.  
\end{theorem}

We note that while the lower bound on the Lipschitz norm of $B$ given by \eqref{eq:B:exp:growth} behaves as $e^t$, in \eqref{eq:nablaB:bound} we have obtained an upper bound which behaves as $e^{C t^{1/2}}$ as $t\to\infty$; this difference stems from the fact that \eqref{eq:B:exp:growth} holds for $\gamma =0$, while 
\eqref{eq:nablaB:bound} holds for $\gamma > d/2+1 = 5/2$. This indicates a different behavior between the MRE equation ($\gamma = 0$ in \eqref{eq:MRE}), and the regularized MRE equation ($\gamma > 0$ is large).

\begin{proof}[Proof of Theorem~\ref{thm:exponential:growth}]
With $v$ as defined in \eqref{eq:hyperbolic:flow}, the $g$  equation \eqref{eq:g:evo} becomes
\begin{align}
\label{eq:g:hyperbolic}
    \p_t g &= \sin^2 x_1 \cos^2 x_2  \p_1^2 g + \cos^2 x_1 \sin^2 x_2 \p_2^2 g - \tfrac{1}{2} \sin(2 x_1) \sin(2 x_2) \p_1 \p_2 g \notag\\
    &\qquad + \tfrac{1}{2} \sin(2 x_1) \p_1 g + \tfrac{1}{2} \sin (2 x_2) \p_2 g.
 \end{align}
The upper bound in \eqref{eq:B:exp:growth} follows from the energy identity
$$
\frac 12 \frac{d}{dt} \| g \|_{L^2}^2  + \| (v \cdot \nabla) g \|_{L^2}^2 = 0 
\,,
$$
the $\dot{H}^k$ estimate for \eqref{eq:g:evo}
$$
\frac 12 \frac{d}{dt} \| g \|_{\dot{H}^k}^2  + \| (v \cdot \nabla) \nabla^k g \|_{L^2}^2 \lesssim \| v \|_{H^k}\| v \|_{H^{k+1}} \| g \|_{H^k}^2
\les \| g \|_{H^k}^2
\,,
$$
and the fact that the condition $k>2$ implies by the 2D Sobolev embedding that $H^k \subset {\rm Lip}$.  

In order to obtain the lower bound in \eqref{eq:B:exp:growth}, we assume without loss of generality that $\p_1 g_0(0,0) \neq 0$ (the case $\p_2 g_0(0,0) \neq 0$ is treated in the same way). 
Then, we differentiate  \eqref{eq:g:hyperbolic} with respect to $x_1$  and arrive at the equation
\begin{equation}
\label{eq:g:exp:growth}
    \p_t (\p_1 g)( 0, 0, t) = (\p_1 g) (0 , 0 ,t)  
    \,,
\end{equation}
The exponential growth for the gradient of $g$, and hence of $B$ in view of \eqref{eq:3D:ansatz}, with respect to the supremum norm now directly follows.
\end{proof}

\section{Open problems}
\label{sec:open:problems}
We conclude the paper by highlighting a number of interesting open problems for the magnetic relaxation equation.

\subsection{Global well-posedness vs finite time blowup}
We have shown in Theorem~\ref{thm:local} that local existence and uniqueness of strong solutions for the MRE equations~\eqref{eq:MRE} holds irrespective of the regularization parameter $\gamma \geq 0$. However, we were only able to prove that these solutions remain smooth for all time (necessary in order to consider the relaxation in the infinite time limit), when $\gamma$ was sufficiently large, namely $\gamma > d/2+1$. Naturally, one is left to consider:
\begin{itemize}
\item [{\bf Q1.}]  For $\gamma \in [0, d/2+1]$ can the local smooth solutions to the {\em active vector} equation \eqref{eq:MRE} be extended to global ones, or do finite time singularities arise?
\end{itemize}
We emphasize that for $d=2$ and for initial magnetic field $B_0$ of zero mean, we may identify a zero mean {\em scalar} magnetic stream function $\phi =  \Delta^{-1} \nabla^\perp \cdot B$,  so that $B = \nabla^\perp \cdot \phi$. Then, the evolution equation \eqref{eq:B:evo} becomes the active scalar equation
\begin{align}
\p_t \phi + u \cdot \nabla \phi = 0
\label{eq:MRE:2d:a} 
\end{align}
where the constitutive law $\phi \mapsto u$ is given by 
\begin{align}
 u = (-\Delta)^{-\gamma} \Proj \div \big(\nabla^\perp \phi \otimes \nabla^\perp \phi \big)  
 \,.
 \label{eq:MRE:2d:b} 
\end{align}
 The active scalar equation \eqref{eq:MRE:2d:a}--\eqref{eq:MRE:2d:b} has a quadratic in $\phi$ constitutive law, and thus a cubic nonlinearity, making the analysis of the 2d MRE equation  more cumbersome when compared to classical models in the canon of active scalar equations, such as SQG~\cite{ConstantinMajdaTabak94} or IPM~\cite{CastroCordobaGancedoOrive09}.

\subsection{Magnetic relaxation}
Assuming that the answer to {\bf Q1} is positive, i.e., that the smooth solutions to  \eqref{eq:MRE} are global in time (and thus the evolution \eqref{eq:B:evo} truly is topology preserving), the fundamental question is whether as $t\to \infty$ the magnetic field $B(\cdot,t)$ {\em relaxes} to an MHD/Euler equilibrium $\bar B$ satisfying \eqref{eq:steady:MHD}. We have discussed in Remark~\ref{rem:relaxation?} the fact that the convergence of $u(\cdot,t) \to 0$ as $t\to \infty$, even with respect to very strong norms, is in general not sufficient to guarantee that weak $L^2$ subsequential limits of $B(\cdot,t)$ are magnetostatic equilibria, i.e., that they solve \eqref{eq:steady:MHD}. Also, in view of Remark~\ref{rem:good:bounds?} and of the results in Section~\ref{sec:3D:instability} we have shown that {\em generically} we cannot expect magnetic relaxation with respect to strong norms, such as $H^1$ or ${\rm Lip}$. Thus, we are naturally lead to:
\begin{itemize}
\item [{\bf Q2.}]  Given a global in time solution $B(\cdot,t)$ to \eqref{eq:MRE}, does there exist a weak solution $\bar B \in L^2(\TT^d)$ of \eqref{eq:steady:MHD} and a subsequence $t_k \to \infty$ such that we have the weak convergence $B(\cdot,t_k) \rightharpoonup \bar B$? Furthermore, does the answer change from $d=2$ to $d=3$, or for various values of  $\gamma$?
\end{itemize}
Furthermore, in~\cite[Section 8, Question (vii)]{Moffatt21} Moffatt poses the question:
\begin{itemize}
\item [{\bf Q3.}]  For the magnetic relaxation problem~\eqref{eq:MRE}, when the initial field is chaotic, what is the asymptotic structure of the relaxed field? Equivalently, what is the function space within which this relaxed field $\bar B$ must reside?
\end{itemize}
The examples given in Section~\ref{sec:3D:instability} show that for the two-and-a-half dimensional solutions constructed via \eqref{eq:3D:ansatz}--\eqref{eq:g:evo}  the answer to {\bf Q2} is positive (see Remark~\ref{rem:asymptotic:L2}), but that the magnetostatic equilibria $\bar B$ may contain current sheets (see Remark~\ref{rem:current}), so that they are not smooth. In fact, by the maximum principle it follows that  all examples of steady states $\bar B$ arising as infinite time limits from the ansatz \eqref{eq:3D:ansatz}--\eqref{eq:g:evo} will lie in $L^\infty$ due to the maximum principle, but our examples show that $\bar B$ cannot in general be expected to lie in $C^0$. We note also that for generic initial data it remains an open problem to show that the MRE evolution \eqref{eq:MRE} is such that $\norm{B(\cdot,t)}_{L^p}$ remains uniformly bounded in time, for any $p>2$. Thus, in general we do not  know if the answer to question {\bf Q3} is better than $\bar B \in L^2$.

\subsection{Global weak solutions}

In the absence of a positive answer to {\bf Q1}, one may wonder whether the MRE system \eqref{eq:MRE} at least possesses global weak solutions for generic initial data. Note that in order to define weak solutions a  minimal requirement is that $B(\cdot,t) \in L^2_x$, in order to properly define a distribution $u$ via \eqref{eq:B:to:u}. However, defining weak solutions to \eqref{eq:B:evo} additionally requires that $(B\otimes u - u \otimes B) \in L^1_{{\rm loc},t,x}$, which is for instance true when $B(\cdot,t) \in L^2$ if we also know that  $u(\cdot,t) \in L^2_x$. Note that when $\gamma$ is not large (e.g. for $\gamma =0$), we cannot deduce from the square integrability of $B$ and \eqref{eq:B:to:u} that $u$ is also square integrable (by the Sobolev embedding this would require $\gamma > d/4 + 1/2$). The energy inequality \eqref{eq:energy:a} comes to the rescue, providing for any $\gamma\geq 0$ the required square integrability in space, locally in time, for the velocity field. A natural question thus is:

\begin{itemize}
\item [{\bf Q4.}]  Given $B_0 \in L^2$ and $\gamma \geq 0$ does there exist a global weak solution $B\in L^\infty_t L^2_x$ of \eqref{eq:MRE} which satisfies the energy inequality \eqref{eq:energy:a}? Alternatively, for $\gamma \in (d/4 + 1/2, d/2+1]$, does there exist a global weak solution $B\in L^\infty_t L^2_x$ of \eqref{eq:MRE} which does not satisfy energy inequality \eqref{eq:energy:a}?  
\end{itemize}
The first part of question {\bf Q4} is nontrivial because the dissipative term in \eqref{eq:energy:a} does not yield robust compactness properties for the vector field $B$. Note, however, that Brenier~\cite{Brenier14}, in the two dimensional case and with $\gamma= 0$, managed to obtain  global in time measure-valued solutions, a notion of solution which is weaker than the one of a weak solution, but which retains a weak-strong uniqueness property.  Concerning the second part of question {\bf Q4}, we note that the cubic nature of the nonlinear term in \eqref{eq:B:evo} and the geometric properties of the constitutive law \eqref{eq:u:def}, prevent the immediate application of convex integration techniques to the MRE system. Indeed, both the $L^\infty$-based convex integration techniques of De Lellis-Szekelyhidi~\cite{DeLellisSzekelihidi12} and the $L^2$-based intermittent convex integration developed by Buckmaster and the third author~\cite{BuckmasterVicol21}, do not seem to be directly applicable to the evolution equation \eqref{eq:MRE}, so that a potentially different convex integration method would need  to be developed for the MRE system.

\subsection{Other models}
A number of other topology preserving diffusion equations have been proposed in the literature, which all have the property that the steady states are incompressible Euler equilibria. We mention for instance the models of Vallis-Carnevale-Young~\cite{VallisCarnevaleYoung89} and Bloch-Marsden-Ratiu~\cite{BlochMarsdenRatiu96}. Other types of coercive damping mechanisms, which, however, do not preserve the topology of the streamlines, were considered in~\cite{Nishiyama01,Nishiyama03,PasqualottoThesis20}. Most if not all of the questions considered in this paper (global existence of solutions, relaxation towards Euler steady states as $t\to \infty$) could be asked about those models. It would be interesting to compare (analytically or numerically) the long-time properties of solutions to the models in \cite{VallisCarnevaleYoung89} or~\cite{BlochMarsdenRatiu96}, with those for the MRE equation. Is there any one model better suited for magnetic relaxation?

%%%%%%%%%%%%%%%%%%%%%%%%%% Bibliography %%%%%%%%%%%%%%%%%%%%%%%%%%%%%%%%%%
\bibliographystyle{alpha}
\bibliography{VladBib}

%%%%%%%%%%%%%%%%%%%%%%%%%% Erase %%%%%%%%%%%%%%%%%%%%%%%%%%%%%%%%%%

% 
%\section{Old Introduction}
%
%
%
%\begin{itemize}
%\item  Introduction: 
%\begin{itemize}
%\item main motivation is Moffatt
%\item idea is to reach Euler steady states via dynamics which preserves Kelvin circulation/topology
%\item mention that topology is preserved irrespective of the ``regularization'' in the constitutive law $B \mapsto u$
%\item mention Brenier's very very weak solutions in 2D, and the fact that not even local existence of strong solutions is known
%\item just mention one sentence about the many ``geometric Euler" papers; Arnold, Khesin, ... Marsen, Ratiu...
%\item  the explicit examples (3D shear flows) showing how Euler/MHD dynamics creates time oscillations (wave equation), but in MRE we have temporal decay and convergence back to the steady state (heat equation)
%\item draw analogy with Tarek's paper (see also  Castro–C\'ordoba–Lear)
%\item mention works of Fanghua et all for MHD; Alf\'en waves
%\end{itemize}
%\item Results:
%\begin{itemize}
%\item local existence
%\item global existence for ``heavy regularization''; mention that as $t\to \infty$ we do not know whether in general we converge to a steady Euler weak solution, or a weak ``subsolution'' (Choffrut-Szekelyhidi)
%\item asymptotic stability on the two dimensional torus of the steady state $B=(1,0)$, $u=(0,0)$
%\end{itemize}
%
%\end{itemize}

\end{document}